\documentclass[final,hidelinks,onefignum,onetabnum]{siamart251104}

\usepackage{mathtools,amssymb}

\usepackage{graphicx}
\usepackage{epstopdf}
\usepackage{algorithmic}
\ifpdf
\DeclareGraphicsExtensions{.eps,.pdf,.png,.jpg}
\else
\DeclareGraphicsExtensions{.eps}
\fi

\crefformat{equation}{\textup{#2(#1)#3}}
\Crefformat{equation}{\textup{#2(#1)#3}}
\crefrangeformat{equation}{\textup{#3(#1)#4--#5(#2)#6}}
\Crefrangeformat{equation}{\textup{#3(#1)#4--#5(#2)#6}}
\crefmultiformat{equation}{\textup{#2(#1)#3}}{ and \textup{#2(#1)#3}}{, \textup{#2(#1)#3}}{, and \textup{#2(#1)#3}}
\Crefmultiformat{equation}{\textup{#2(#1)#3}}{ and \textup{#2(#1)#3}}{, \textup{#2(#1)#3}}{, and \textup{#2(#1)#3}}
\crefrangemultiformat{equation}{\textup{#3(#1)#4--#5(#2)#6}}{ and \textup{#3(#1)#4--#5(#2)#6}}{, \textup{#3(#1)#4--#5(#2)#6}}{, and \textup{#3(#1)#4--#5(#2)#6}}
\Crefrangemultiformat{equation}{\textup{#3(#1)#4--#5(#2)#6}}{ and \textup{#3(#1)#4--#5(#2)#6}}{, \textup{#3(#1)#4--#5(#2)#6}}{, and \textup{#3(#1)#4--#5(#2)#6}}

\newsiamremark{remark}{Remark}
\newsiamremark{hypothesis}{Hypothesis}
\crefname{hypothesis}{Hypothesis}{Hypotheses}
\newsiamthm{claim}{Claim}
\newsiamremark{fact}{Fact}
\crefname{fact}{Fact}{Facts}

\headers{A Note on the Point-Clothoid Distance Algorithm}{H.~Ye, H.~Ge, and G.~Cheng}

\title{A Note on the Point-Clothoid Distance Algorithm}

\author{Haibin Ye\thanks{H.~Ye and H.~Ge are co-first authors.}
\and Hao Ge\footnotemark[1]
\and Gong Cheng\thanks{Corresponding author (\email{gongch@tongji.edu.cn}).}}

\usepackage{longtable}
\usepackage{subcaption}

\ifpdf
\hypersetup{
  pdftitle={Point-Clothoid Distance},
  pdfauthor={H.~Ye, H.~Ge, G.~Cheng}
}
\fi

\begin{document}

\maketitle

\begin{abstract}
  Computing the closest point on a clothoid is a recurring task in geometric design, road and railway alignment, and path planning. The efficient algorithm of Frego and Bertolazzi addresses this problem, but its candidate-selection analysis assumes at most one local minimum per search interval. We exhibit admissible configurations with two local minima, raising the question of whether the existing strategy accounts for every possible minimum. Using the geometry of the clothoid evolute, we prove that, for any query point and any proper no-inflection planar clothoid segment with tangent-angle variation at most $2\pi$, the squared-distance function has at most three stationary points; if all three are local extrema, their order is min-max-min. This establishes the completeness of the original candidate-selection logic beyond the one-minimum premise. It also shows that no interior search is needed when neither endpoint derivative test is active, allowing unnecessary midpoint searches to be omitted while retaining numerical fallback. Numerical experiments demonstrate reductions in iteration count and evaluation time.
\end{abstract}



\section{Introduction}

Clothoids are planar curves whose curvature varies linearly with arc length. Their smooth curvature transitions make them useful in geometric design, road and railway alignment, and curvature-constrained path planning \cite{VAZQUEZMENDEZ2024geometric,meek1992clothoid,bakolas2009path}. These applications repeatedly require solving the point-clothoid distance problem: given a query point $q$, determine its closest point on a clothoid segment.

Frego and Bertolazzi~\cite{frego2019point} developed a comprehensive and efficient algorithm for this problem. Their method exploits the geometry of clothoids to restrict the search interval and employs a Newton-like iteration with carefully selected initial candidates. It provides a highly effective approach for point-clothoid distance and projection computation.

The candidate-selection strategy of \cite{frego2019point} is supported by an analysis that assumes at most \emph{one local minimum} candidate on each no-inflection interval with tangent-angle variation less than $2\pi$. We exhibit an admissible configuration with two local minima. We then prove the sharp upper bound that, for every proper no-inflection segment with tangent-angle variation at most $2\pi$, the squared-distance function has at most three stationary points; if all three are local extrema, they occur in min--max--min order. This count and classification provide a complete justification for the candidate-selection logic of Algorithm~A.7 in \cite{frego2019point} (hereafter, the \emph{distance algorithm}) without the one-minimum assumption, including the two-minimum case. Moreover, when neither endpoint derivative tests activates an endpoint search, we prove the no interior local minimum exists. Hence the midpoint is unnecessary as an additional candidate in this derivative-based branch. This geometric redundancy should be distinguished from the midpoint's seperate role as a numerical fallback when an activated endpoint iteration fails; our analysis does eliminate that safeguarding role.


To establish the stationary-point bound, we identify stationary points of the squared-distance function with tangent lines of the clothoid evolute passing through the query point. We divide the evolute into four segments with tangent-angle variation $\pi/2$, represent each as the graph of a strictly convex function in suitable affine frames, and combine the segment-wise bounds while correcting for shared endpoints. The same framework covers the endpoint-degenerate case of zero initial curvature.

The remainder of this manuscript is organized as follows.
\Cref{sec:Clothoid and Evolute} introduces the clothoid and its evolute and
establishes the correspondence between clothoid normals and evolute tangents.
\Cref{sec:main} proves the sharp stationary-point bound, classifies the
three-extrema case, and derives a reduced candidate-selection strategy with numerical fallback.
\Cref{sec:experiments} presents the numerical comparison of the original and
streamlined algorithms. Finally, \Cref{sec:conclusions} summarizes the main
results and their algorithmic implications.

\section{Clothoid and Its Evolute}
\label{sec:Clothoid and Evolute}
Let $\boldsymbol{p}(s)=(x(s),y(s))^{\mathsf{T}}$, $s\in[0,L]$, be an
arc-length-parametrized clothoid with tangent angle
$\theta(s)=ks^2/2+k_0s+\theta_0$ and curvature
$\kappa(s)=\theta'(s)=ks+k_0$. Then
\begin{equation}
  \label{eq:clothoid_arc_parametrization}
  \boldsymbol{p}(s)=\boldsymbol{p}(0)+
  \int_0^s
  \begin{pmatrix}\cos\theta(\tau)\\ \sin\theta(\tau)
  \end{pmatrix}
  \,\mathrm{d}\tau.
\end{equation}
A translation and rotation allow us to set
$\boldsymbol{p}(0)=\boldsymbol{0}$ and $\theta_0=0$ without loss of
generality. We call the clothoid proper when $k\neq0$; for $k=0$, it is a
circle or a line, as treated in \cite{frego2019point}.

Let $\mathbf{T}(s)=(\cos\theta(s),\sin\theta(s))^{\mathsf{T}}$ and
$\mathbf{N}(s)=(-\sin\theta(s),\cos\theta(s))^{\mathsf{T}}$. Wherever
$\kappa(s)\neq0$, the evolute, i.e., the locus of centers of curvature, is
$\boldsymbol{e}(s)=\boldsymbol{p}(s)+\mathbf{N}(s)/\kappa(s)$; under
the above normalization, its coordinates are
\begin{equation}
  \label{eq:evolute_arc_parametrization_with_k0}
  x_e(s)=\int_0^s\cos\theta(\tau)\,\mathrm{d}\tau
  -\frac{\sin\theta(s)}{\kappa(s)},\qquad
  y_e(s)=\int_0^s\sin\theta(\tau)\,\mathrm{d}\tau
  +\frac{\cos\theta(s)}{\kappa(s)}.
\end{equation}
\begin{figure}[htbp]
  \centering
  \includegraphics[width=0.4\textwidth]{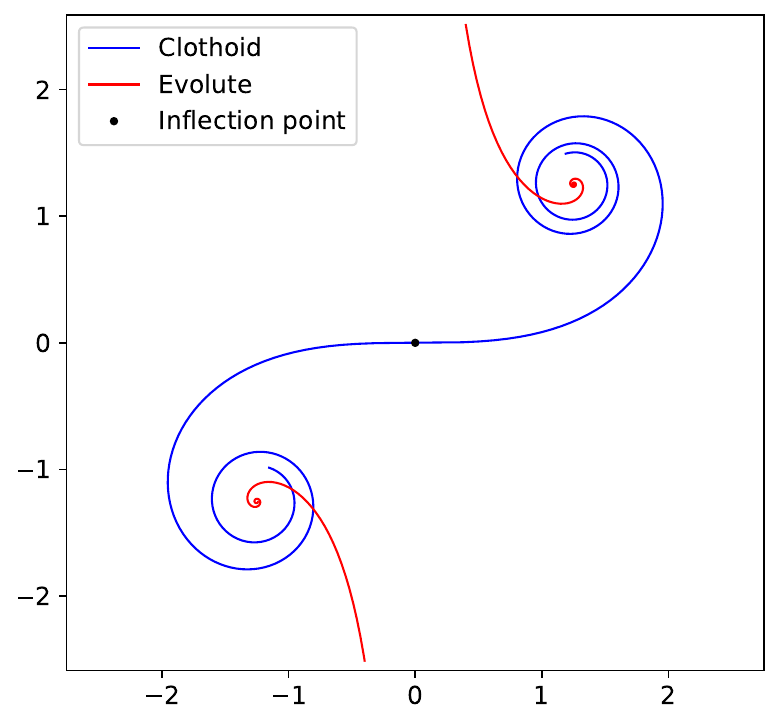}\hspace{1.5em}
  \includegraphics[width=0.312\textwidth]{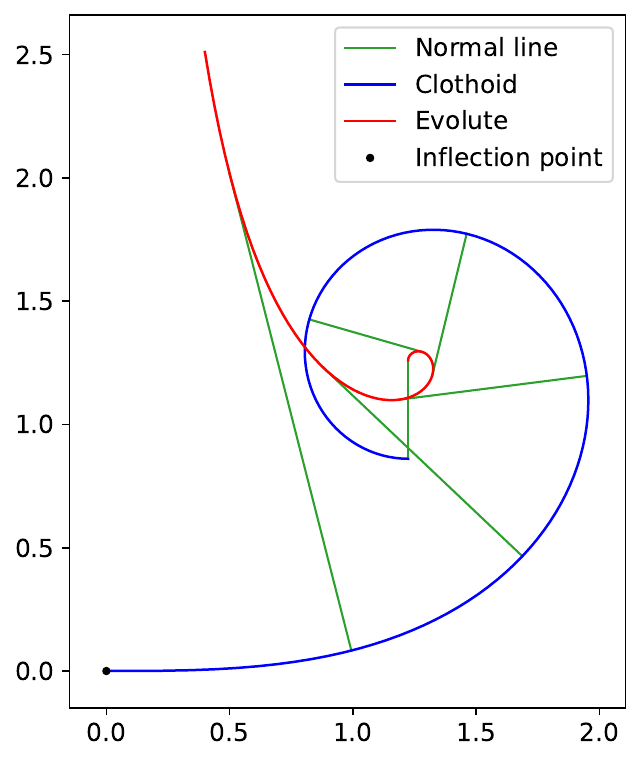}
  \caption{The clothoid (blue) and its evolute (red). Left: the full curve and
    the two evolute branches separated by the inflection point. Right: a
    no-inflection branch with tangent-angle variation $2\pi$; the green lines
  are clothoid normals and evolute tangents at corresponding parameters.}
  \label{fig:clothoid-evolute}
\end{figure}

Since $\boldsymbol{e}'(s)=-k\mathbf{N}(s)/\kappa(s)^2$ wherever
$\kappa(s)\neq0$, the tangent line of evolute at $\boldsymbol{e}(s)$
is the normal line of clothoid at $\boldsymbol{p}(s)$. For $k_0=0$, the
tangent at $s=0$ is understood as a limit. Thus, for any
query point $\boldsymbol{q}$, counting clothoid normal lines through
$\boldsymbol{q}$ is equivalent to counting evolute tangents
through $\boldsymbol{q}$.

\section{Stationary Points of Point-Clothoid Distance Function}
\label{sec:main}

The distance algorithm computes the distance from a query point to a standard clothoid segment with no interior inflection and tangent-angle variation less than $2\pi$. It uses derivative information at the two endpoints to determine which endpoint Newton-like iterations should be attempted. At the candidate-selection level, the midpoint provides an additional initialization when neither endpoint derivative test activates. The same midpoint initialization may also serve as a numerical fallback when at least one endpoint search is activated but none of the activated iterations succeeds. The above strategy is supported by the conclusion in Section~3
of \cite{frego2019point} that such an interval contains at most
one local minimum.
\Cref{fig:counterexample_minmaxmin} nevertheless exhibits an admissible configuration with three stationary points, including two local minima. This raises the question of whether the candidate-selection logic remains complete and motivates a sharp count and classification of the stationary points.
\begin{figure}[htbp]
  \centering
  \includegraphics[width=0.4\textwidth]{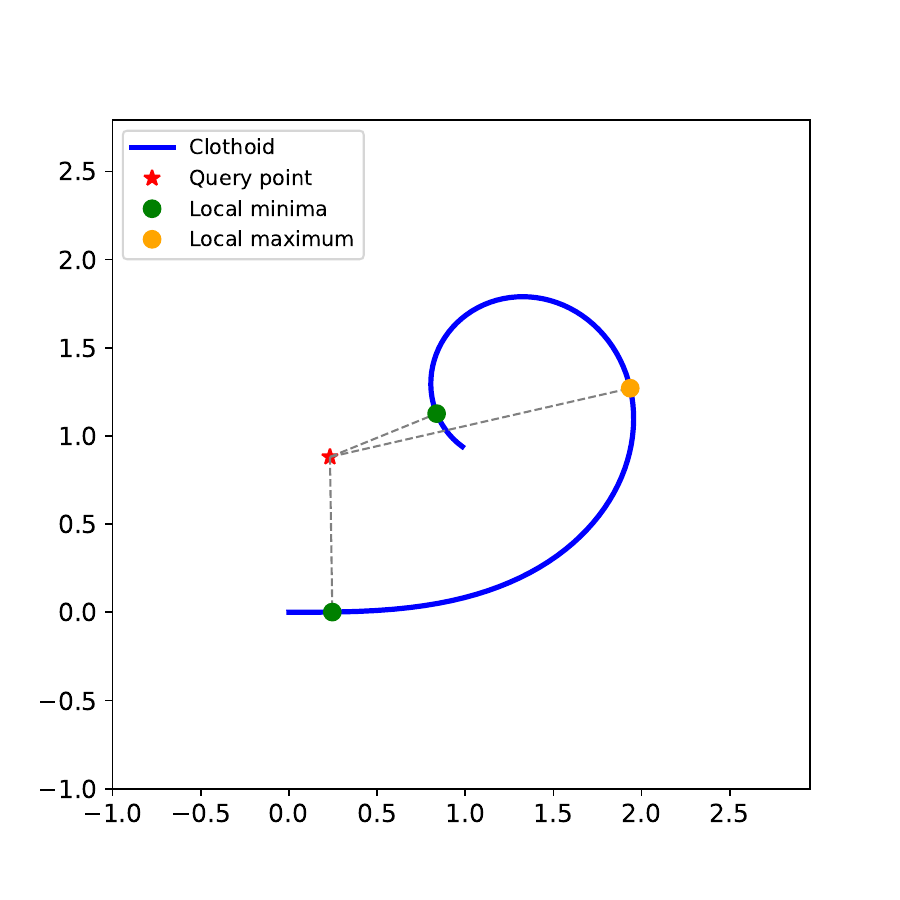}\hspace{0.5em}
  \includegraphics[width=0.4\textwidth]{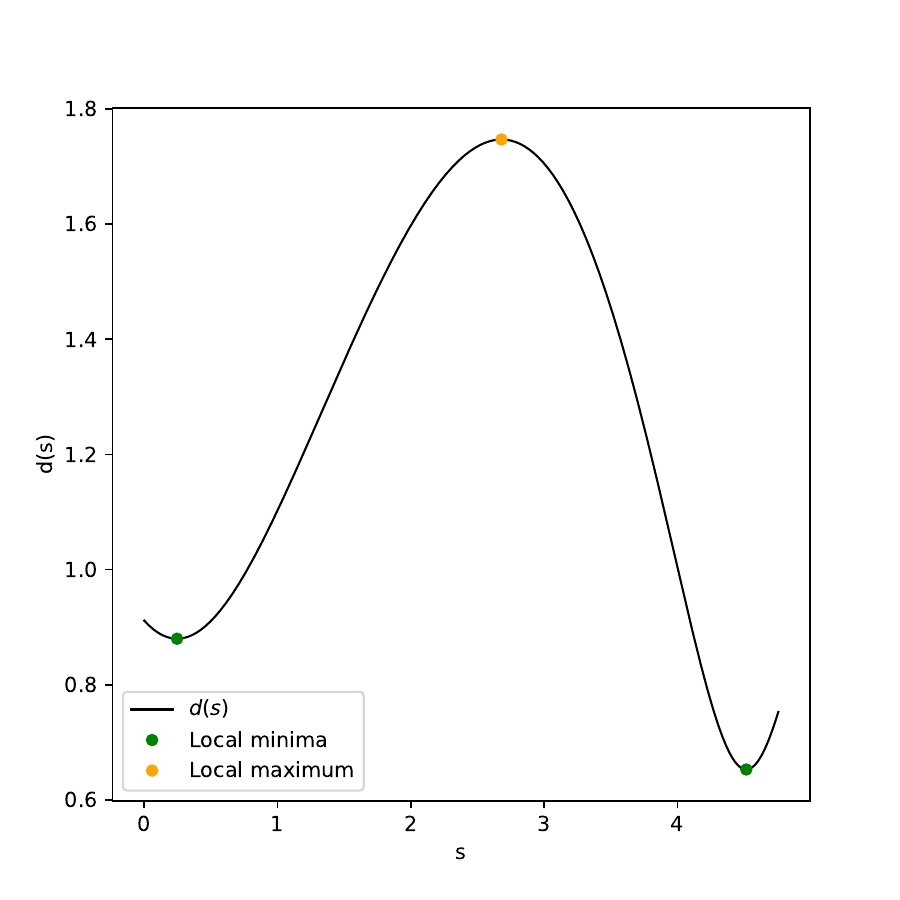}
  \caption{Left: the clothoid, the query point $\boldsymbol{q}=(0.233,0.881)^{\mathsf{T}}$ and the normals at the three stationary points. Right: the distance function $d(s)$. On a clothoid segment with $k=\frac12$, $k_0=0$ and angle variation $1.8\pi<2\pi$, the distance function has three stationary points: two local minima (green) and one local maximum (orange).}
  \label{fig:counterexample_minmaxmin}
\end{figure}

Let $d(s)=\|\boldsymbol{p}(s)-\boldsymbol{q}\|$ and define the smooth
squared-distance function
\begin{equation}
  \label{eq:squared-distance}
  D(s)=\frac{1}{2}d(s)^2,
  \quad
  D'(s)=(\boldsymbol{p}(s)-\boldsymbol{q})^{\mathsf{T}}\mathbf{T}(s).
\end{equation}
For a proper no-inflection segment with $\Delta\theta<2\pi$, Lemmas
3.1--3.3 of \cite{frego2019point} show that stationary points
which are local extrema are simple and isolated, whereas multiple zeros
of $D'$ are stationary inflection points of $D$. In particular, the local
extrema alternate. We now establish the required global bound on the
number of stationary points.

It suffices to consider the normalized $2\pi$ case. The degenerate case
$k=0$, a circle or line, is treated in Section~4 in
\cite{frego2019point}. For $k\neq0$, reversing the parametrization
and, if necessary, reflecting the plane preserve distances and tangency
multiplicities, so we may assume $k>0$ and $k_0\ge0$. Any no-inflection
segment with $\Delta\theta<2\pi$ can then be extended terminally to
$\Delta\theta=2\pi$ without losing stationary points. Finally,
\Cref{eq:squared-distance} shows that $s_\star$ is stationary exactly when
the normal line at $\boldsymbol{p}(s_\star)$ passes through
$\boldsymbol{q}$, equivalently when the tangent line of the evolute at
$\boldsymbol{e}(s_\star)$ passes through $\boldsymbol{q}$. The desired
stationary-point bound therefore reduces to the equivalent tangency bound
stated below.
\begin{theorem}[Number of Stationary Points]
  \label{thm:main}
  Let $\boldsymbol{q}\in\mathbb{R}^2$ and let $\boldsymbol{p}(s)$,
  $s\in[0,s_{2\pi}]$, be a proper clothoid segment satisfying
  $k>0$, $k_0\ge0$, and
  $\theta(s_{2\pi})-\theta(0)=2\pi$. Then $D(s)$ has at most three
  stationary points on $[0,s_{2\pi}]$.
\end{theorem}
\addtocounter{theorem}{-1}
\renewcommand{\thetheorem}{\arabic{section}.\arabic{theorem}$^\prime$}
\begin{theorem}
  \label{thm:main-prime}
  Let $\boldsymbol{q}\in\mathbb{R}^2$ and let $\boldsymbol{p}(s)$,
  $s\in[0,s_{2\pi}]$, be a proper clothoid segment satisfying
  $k>0$, $k_0\ge0$, and
  $\theta(s_{2\pi})-\theta(0)=2\pi$. Let $\boldsymbol{e}(s)$ denote its
  evolute wherever $\kappa(s)\neq0$. Then at most three parameters
  $s\in[0,s_{2\pi}]$ have an evolute tangent line passing through
  $\boldsymbol{q}$; when $k_0=0$, the tangent at $s=0$ is the limiting line
  as $s\to0^+$.
\end{theorem}
\renewcommand{\thetheorem}{\arabic{section}.\arabic{theorem}}
We prove \Cref{thm:main-prime}. Since
$\kappa(s)=\theta'(s)=k_0+ks$ with $k>0$ and $k_0\ge0$, the only cases
are $k_0>0$ and the endpoint-degenerate case $k_0=0$; the latter will be
recovered by a limiting argument at the end of the subsection.

\subsection{Number of Evolute Tangents for Nonnegative Initial Curvature}

We begin with the nondegenerate case $k_0>0$.
Let $s_\alpha$ satisfy $\theta(s_\alpha)=\alpha$ and define the four
closed evolute segments
$\Gamma_i=\boldsymbol{e}([s_{(i-1)\pi/2},s_{i\pi/2}])$,
$i=1,\ldots,4$. The evolute and clothoid have the same tangent-angle
variation, so each $\Gamma_i$ spans $\pi/2$. In either affine frame
\[
  [\boldsymbol{e}(s_{(i-1)\pi/2});
    -\mathbf{N}(s_{(i-1)\pi/2}),
  \mathbf{T}(s_{(i-1)\pi/2})]
  \quad\text{or}
  \quad  [\boldsymbol{e}(s_{i\pi/2});
    \mathbf{N}(s_{i\pi/2}),
  \mathbf{T}(s_{i\pi/2})],
\]
the segment is the graph of a function
$f:[0,M]\to\mathbb{R}$, $0<M<\infty$. The frame construction gives
$f(0)=f'(0)=0$ and $\lim_{x\to M^-}f'(x)=\infty$; it remains only to prove
that $f$ is strictly convex on $(0,M)$.
\begin{figure}[htbp]
  \centering
  \includegraphics[width=0.4\textwidth]{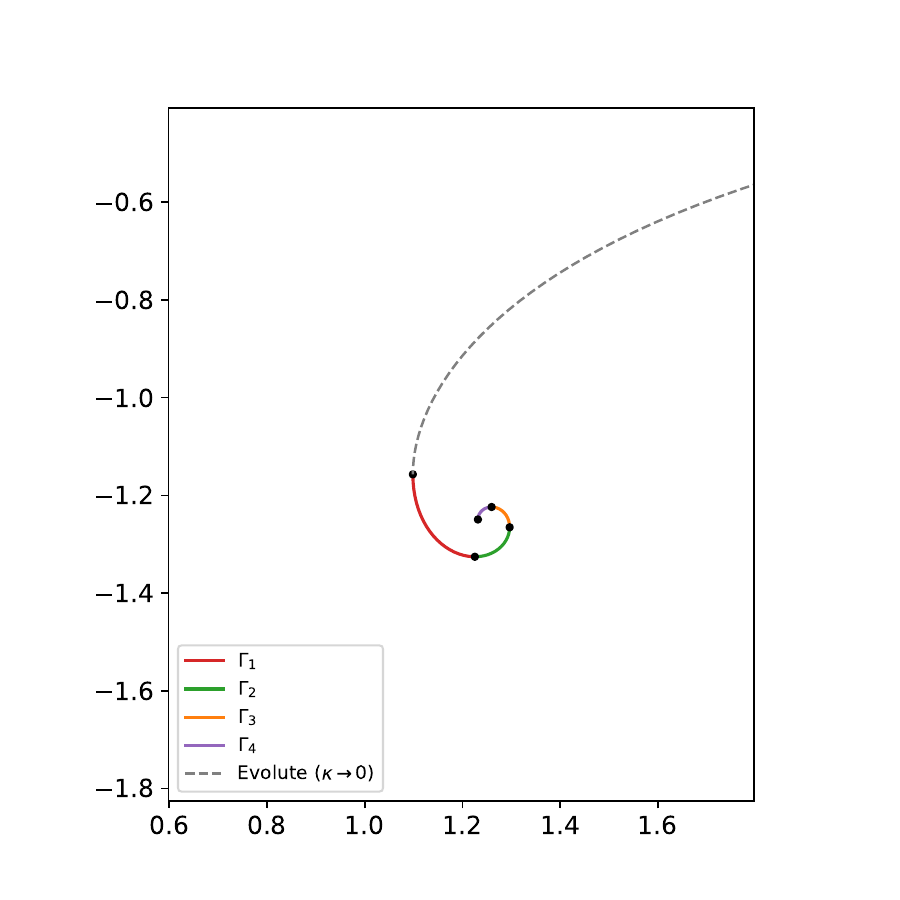}\hspace{0.5em}%
  \includegraphics[width=0.4\textwidth]{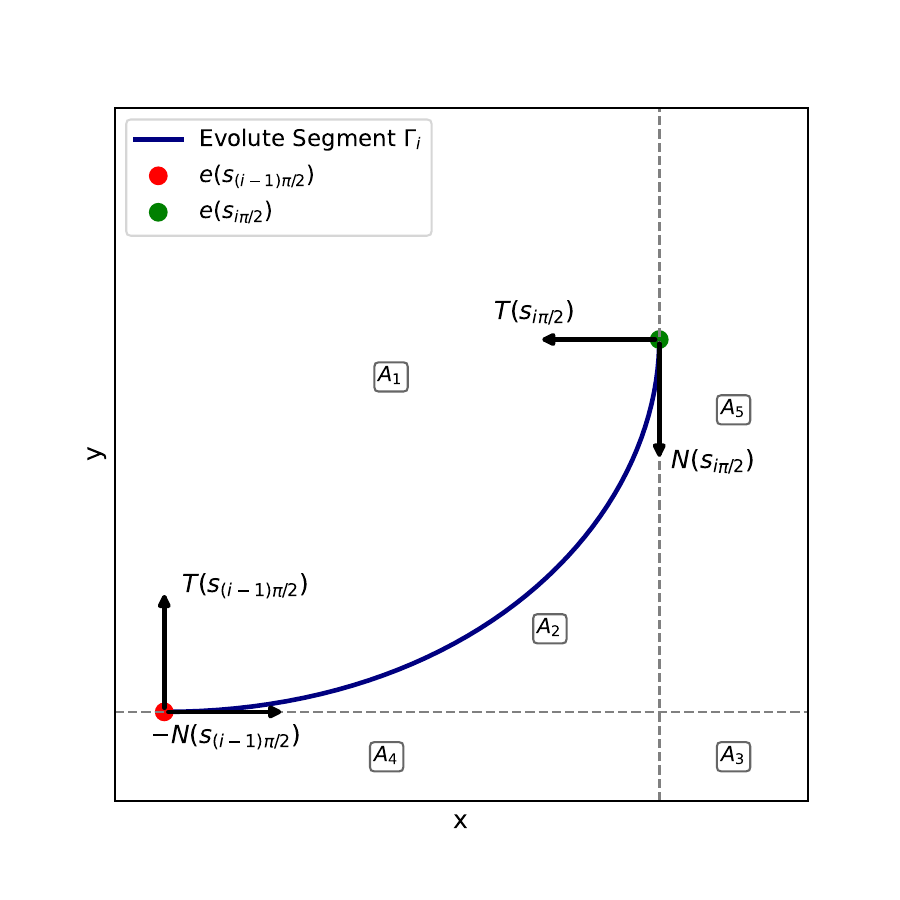}
  \caption{Left: after rotation placing the initial evolute tangent
    parallel to the negative $y$-axis, the evolute is divided at
    $s_\beta$, $\beta=0,\pi/2,\pi,3\pi/2,2\pi$, into four segments
    $\Gamma_1$--$\Gamma_4$; the dashed $\kappa\to0$ branch tends to
    infinity in the positive $x$-direction after rotation.
    Right: $\Gamma_i$ represented as a strictly convex graph in its two
    affine frames. The graph and its endpoint tangents (dashed) partition
  the plane into regions $A_1$--$A_5$.}
  \label{fig:evolute-segmentation}

\end{figure}
\begin{lemma}
  \label{lem:convexity}
  A function $f$ constructed as above is strictly convex on $(0,M)$.
\end{lemma}

\begin{proof}
  For every segment--frame pair, the same calculation follows after shifting
  or reversing the local tangent angle. It therefore suffices to display it
  for $\Gamma_2$ in the frame
  $[\boldsymbol{e}(s_{\pi/2});-\mathbf{N}(s_{\pi/2}),
  \mathbf{T}(s_{\pi/2})]$. Let
  $(x_\gamma(s),y_\gamma(s))$, $s\in[s_{\pi/2},s_\pi]$, denote its
  coordinates. Since $\theta(s)=ks^2/2+k_0s$ and
  $\kappa(s)=ks+k_0$,
  \begin{equation}
    \label{eq:derivative of para_gamma_2}
    x_\gamma'(s)=\frac{k\sin\theta(s)}{\kappa(s)^2},
    \qquad
    y_\gamma'(s)=-\frac{k\cos\theta(s)}{\kappa(s)^2}.
  \end{equation}
  For $s\in(s_{\pi/2},s_\pi)$, $x_\gamma'(s)>0$, and hence
  \begin{equation}
    \label{derivative of f}
    f'(x)=\frac{y_\gamma'(s)}{x_\gamma'(s)}=-\cot\theta(s),
    \qquad
    f''(x)=\frac{\kappa(s)^3}{k\sin^3\theta(s)}>0.
  \end{equation}
  Thus $f$ is strictly convex on $(0,M)$.
\end{proof}

Next, we examine the segment-wise bound of evolute tangents for the five regions in
\Cref{fig:evolute-segmentation}. The following conclusions in \Cref{prop:y<0,prop:y>f(x),prop:all-points,cor:convex-extension} are required for the bound analysis, the proofs of which are left in the appendix.

\begin{proposition}
  \label{prop:y<0}
  Let $a\in\mathbb{R}$, $b\in[-\infty,a)$, and
  $f\in C^2((b,a))$ satisfy
  \[
    f'(x)>0,\qquad f''(x)>0,\qquad
    \lim_{x\to b^+}f(x)=\lim_{x\to b^+}f'(x)=0,\qquad
    \lim_{x\to a^-}f'(x)=\infty.
  \]
  If $b$ is finite, set $f(b)=0$; if
  $\lim_{x\to a^-}f(x)<\infty$, extend $f$ continuously to $a$.
  Then every point $\boldsymbol{q}=(q_x,q_y)^{\mathsf{T}}$ with $q_y<0$ lies on
  at most one tangent line to the graph of $f$. If also $q_x>a$, no such
  tangent line exists.
\end{proposition}

\begin{proposition}
  \label{prop:y>f(x)}
  Under the assumptions of \Cref{prop:y<0}, if
  $q_x\in[b,a)$ and $q_y>f(q_x)$, no tangent line to the graph of
  $f$ passes through $\boldsymbol{q}$.
\end{proposition}

\begin{corollary}
  \label{cor:convex-extension}
  Suppose $b> -\infty$ and extend $f$ by
  \begin{equation}
    \label{eq:F(x)}
    F(x)=\left\{
      \begin{aligned}
        0, &\quad x<b,\\
        f(x), &\quad b\le x<a.
      \end{aligned}
      \right.
    \end{equation}
    Then $F$ is convex on $(-\infty,a)$. Consequently, if $q_x<a$ and
    $q_y>F(q_x)$, no tangent line to the graph of $F$ passes through
    $\boldsymbol{q}$.
  \end{corollary}

  \begin{proposition}
    \label{prop:all-points}
    Under the assumptions of \Cref{prop:y<0}, at most two tangent lines to
    the graph of $f$ pass through any point $\boldsymbol{q}\in\mathbb{R}^2$.
    If $q_y=0$ and either $b> -\infty$ with $q_x\le b$, or $q_x>a$, at
    most one tangent line passes through $\boldsymbol{q}$.
  \end{proposition}

  Applying \Cref{prop:y<0,prop:y>f(x),prop:all-points,cor:convex-extension}
  in the two endpoint affine frames gives the segment-wise bounds in
  \Cref{tab:tangents}.

  \begin{table}[htbp]
    \footnotesize
    \caption{Segment-wise upper bounds for the number of tangents through a query point
      in the close regions and their boundaries of \Cref{fig:evolute-segmentation}; each
    result is applied in the relevant endpoint affine frame. Each set in this table is understood without its boundary: the rows listing a single region give the cells $A_i$ without their boundaries, the rows listing a two-region intersection give the edges without their endpoints, and the endpoints appear only in the last row.}
    \label{tab:tangents}
    \begin{center}
      \begin{tabular}{|c|c|c|} \hline
        \textbf{Region} & \textbf{Upper bound} & \textbf{Justification} \\ \hline
        $A_1$ & 0 & \Cref{prop:y>f(x)}, in both frames \\
        $A_2$ & 2 & \Cref{prop:all-points} \\
        $A_3$ & 0 & \Cref{prop:y<0} ($q_x > a$, $q_y < 0$) \\
        $A_4$ & 1 & \Cref{prop:y<0}, frame at $\boldsymbol{e}(s_{(i-1)\pi/2})$ \\
        $A_5$ & 1 & \Cref{prop:y<0}, frame at $\boldsymbol{e}(s_{i\pi/2})$ \\ \hline
        $A_1 \cap A_2$ (graph of $f$) & 1 & unique tangent at the point \\
        $A_1 \cap A_4$ or $A_1 \cap A_5$ & 1 & special case of \Cref{prop:all-points} \\
        $A_2 \cap A_4$ or $A_2 \cap A_5$ & 2 & \Cref{prop:all-points} \\
        $A_3 \cap A_4$ or $A_3 \cap A_5$ & 1 & special case of \Cref{prop:all-points} \\ \hline
        $A_2 \cap A_3 \cap A_4 \cap A_5$ & 2 & \Cref{prop:all-points} \\ \hline
      \end{tabular}
    \end{center}
  \end{table}

  To combine these segment-wise bounds, we first determine the relative
  positions of the division points $\boldsymbol{e}(s_\beta)$,
  $\beta=0,\pi/2,\pi,3\pi/2,2\pi$.

  Direct calculation gives
  $y_e(s_0)>y_e(s_\pi)$,
  $x_e(s_{\pi/2})<x_e(s_{3\pi/2})$, and
  $y_e(s_\pi)<y_e(s_{2\pi})$. The remaining comparisons are linked: using
  $\theta$ as the integration variable and pairing the positive and negative
  cosine lobes shows that $x_e(s_{\pi/2})<x_e(s_{2\pi})$ implies
  $y_e(s_0)>y_e(s_{3\pi/2})$.
  Thus only three cases occur: Case A has the two displayed inequalities;
  Case B has both reversed; and Case C has
  $y_e(s_0)>y_e(s_{3\pi/2})$ and
  $x_e(s_{\pi/2})>x_e(s_{2\pi})$. Equality in either comparison introduces
  no additional region, since the resulting boundary sets are already
  covered by the non-strict bounds in \Cref{tab:tangents}.

  Since the segments $\Gamma_i$ are closed, a tangency at an internal
  division parameter is counted in both adjacent segment bounds. Define
  \begin{equation}
    \delta_j(\boldsymbol{q})=
    \begin{cases}
      1, & \boldsymbol{q}\text{ lies on the tangent line at }\boldsymbol{e}(s_{j\pi/2}),\\
      0, & \text{otherwise},
    \end{cases}
    \qquad j=1,2,3.
    \label{eq:junction-correction}
  \end{equation}
  This correction removes only the duplicate occurrence of the same
  division parameter, not tangencies at distinct parameters.

  Superimpose the plane divisions induced by the four segments and let
  $u_i(\boldsymbol{q})$ be the bound for $\Gamma_i$ from
  \Cref{tab:tangents}. We define
  \begin{equation}
    \widehat N(\boldsymbol{q}):=
    \sum_{i=1}^{4}u_i(\boldsymbol{q})-
    \sum_{j=1}^{3}\delta_j(\boldsymbol{q}),
    \label{eq:total-tangent-bound}
  \end{equation}
  which is no less than the actual number of evolute tangents, $N(\boldsymbol{q})$, i.e., $N(\boldsymbol{q})\leq\widehat N(\boldsymbol{q})$.
  \begin{lemma}
    \label{lem:global-tangency-count}
    In all Cases A--C, 
    $\widehat N(\boldsymbol{q})\leq3$ for every
    $\boldsymbol{q}\in\mathbb{R}^2$. Moreover,
    $\widehat N(\boldsymbol{q})=3$ only in the closed strip
    \begin{equation}
      x(0)\leq q_x\leq x(s_{2\pi}).
      \label{eq:three-tangency-strip}
    \end{equation}
  \end{lemma}

  \begin{proof}
    Superimposing the four segment-wise partitions in
    \Cref{fig:plane-division}, applying \Cref{tab:tangents} on every region
    and boundary, and subtracting the junction corrections in
    \cref{eq:junction-correction} gives \cref{eq:total-tangent-bound}; the
    exhaustive enumeration in \Cref{sec:supp-tangency-enumeration} and \Cref{tab:tangents_all_case} of the appendix shows that $\widehat N\le3$, with equality only in the strip
    \cref{eq:three-tangency-strip}, and \Cref{fig:stationary-heatmaps} there gives the corresponding
    sampled counts.
  \end{proof}
  \begin{figure}[htbp]
    \centering
    \setlength{\tabcolsep}{4pt}
    \begin{tabular}{ccc}
      \footnotesize case A & \footnotesize case B & \footnotesize case C \\
      \includegraphics[width=0.31\textwidth]{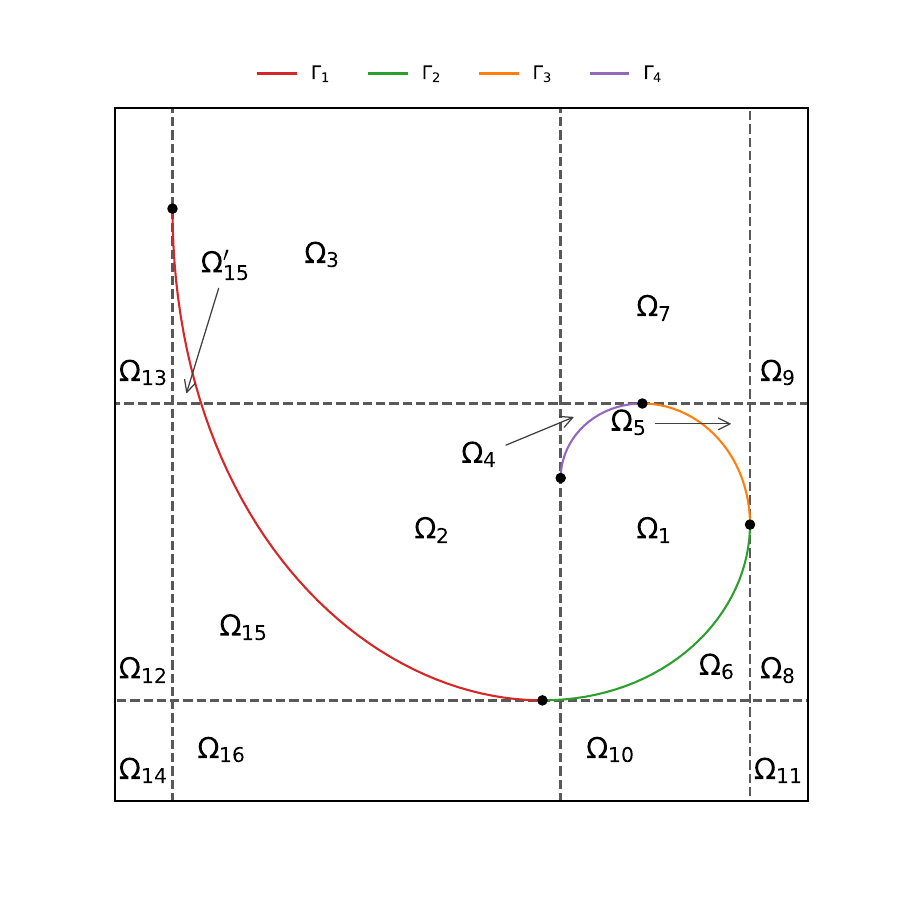}
      & \includegraphics[width=0.31\textwidth]{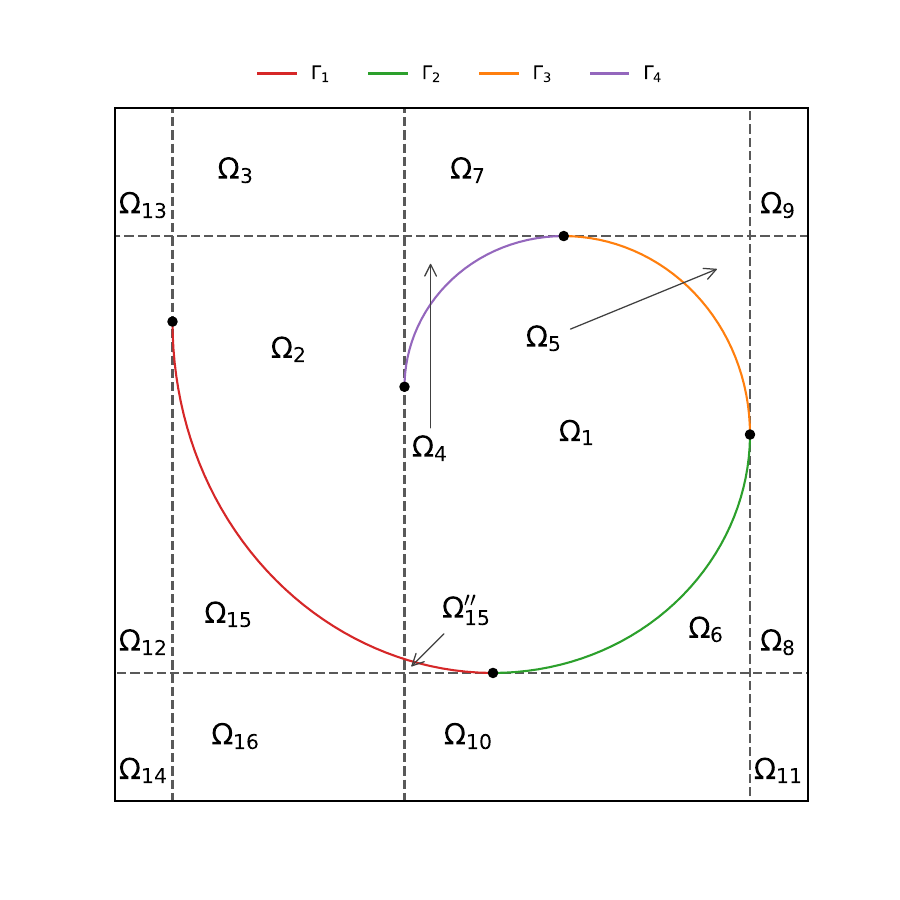}
      & \includegraphics[width=0.31\textwidth]{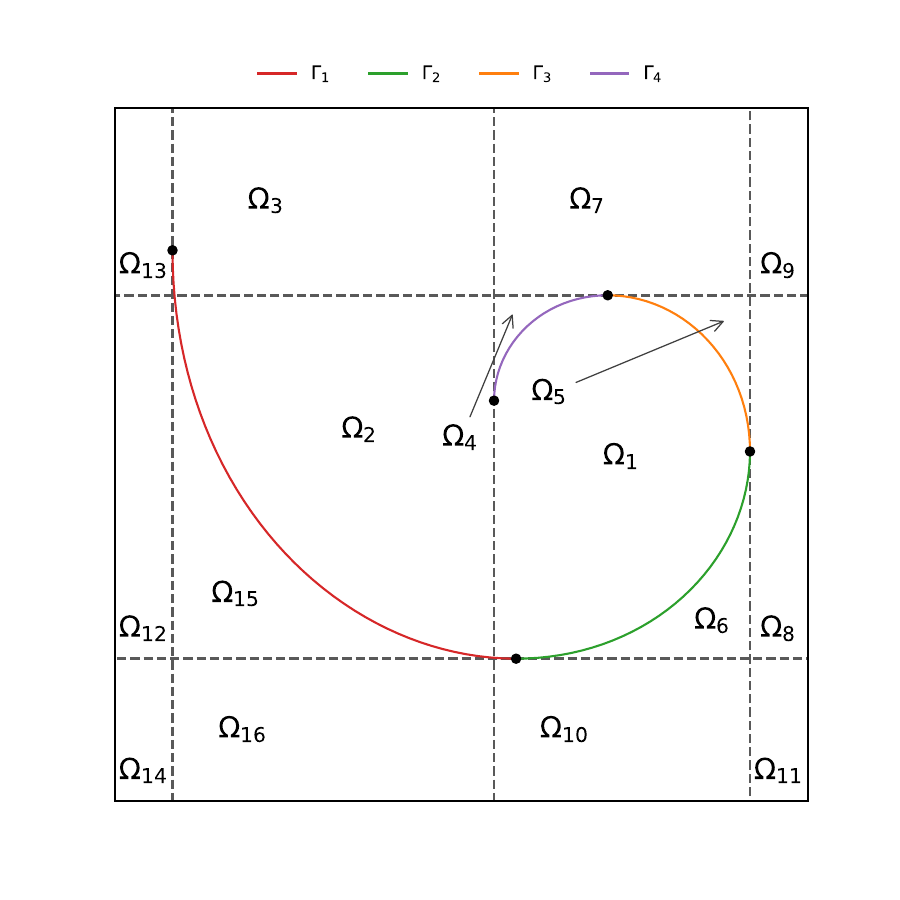} \\
      \includegraphics[width=0.31\textwidth]{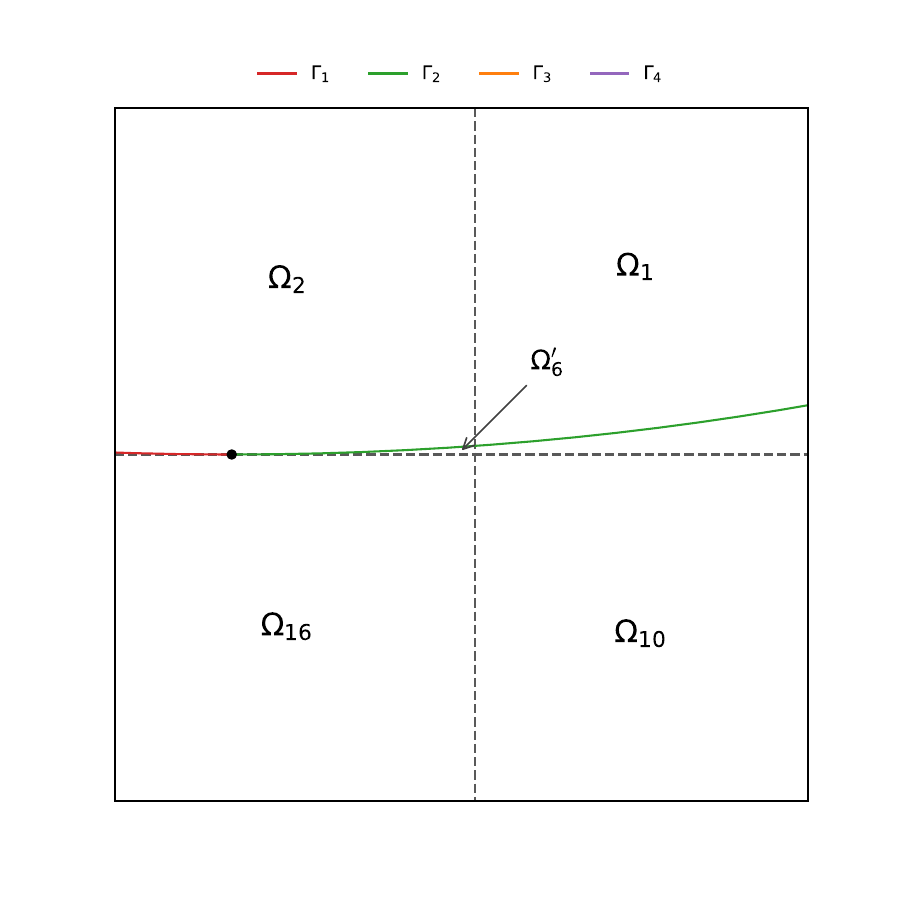}
      & \includegraphics[width=0.31\textwidth]{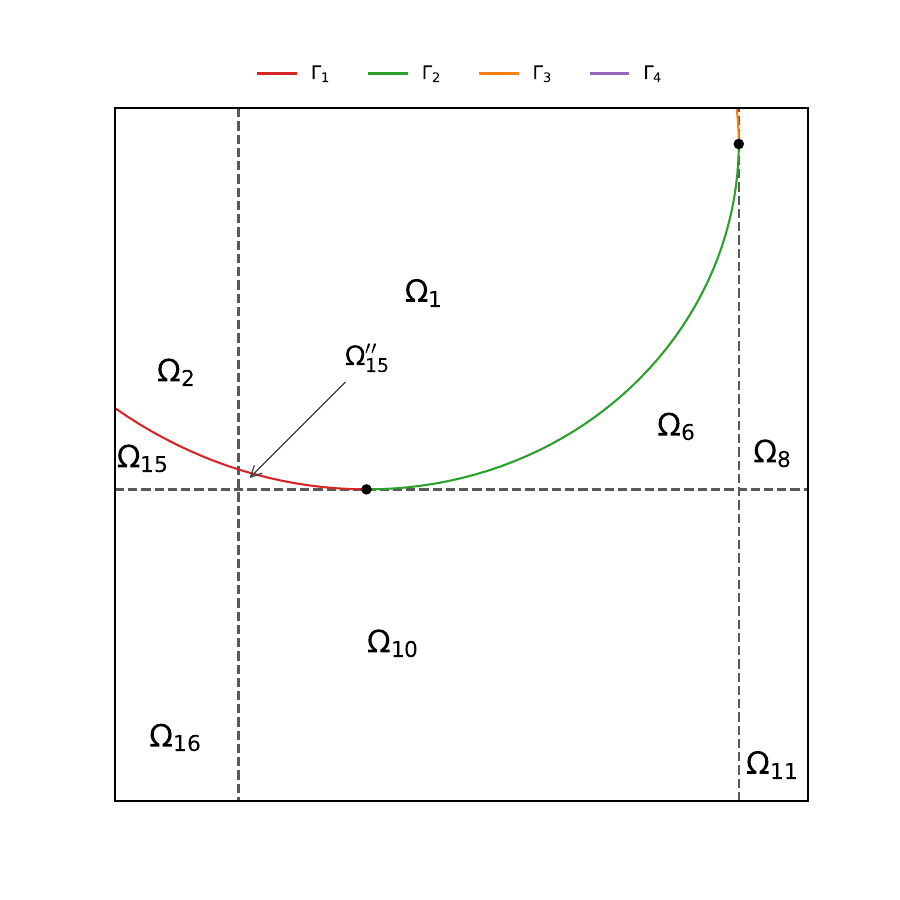}
      & \includegraphics[width=0.31\textwidth]{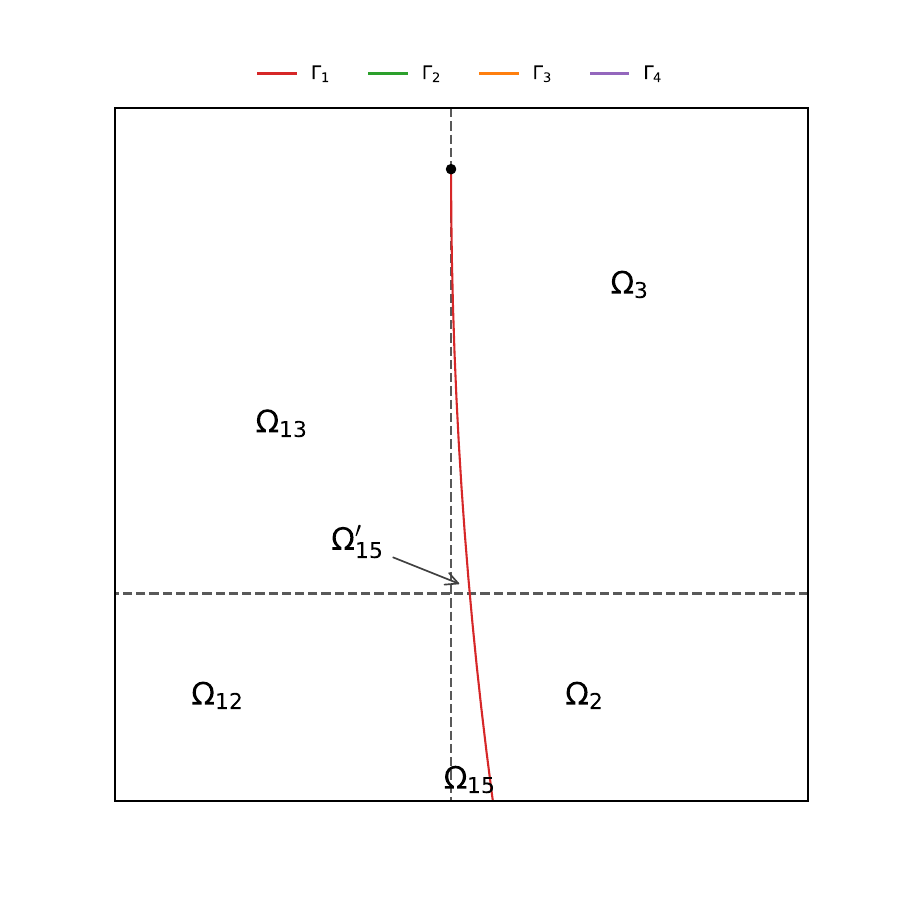}
    \end{tabular}
    \caption{Plane subdivisions for Cases A--C. Dashed tangents at
      $\boldsymbol{e}(s_\beta)$ and colored segments $\Gamma_1$--$\Gamma_4$
      delimit regions labeled by their tangency upper bounds. Top: full
      subdivisions; bottom: arrowed enlargements. The smallest regions,
    $\Omega_4$ and $\Omega_5$, occur in all three cases.}
    \label{fig:plane-division}
  \end{figure}

  For the degenerate case where the initial curvature $k_0$ vanishes, the tangency correspondence and four-segment division remain
  valid when the initial tangent is interpreted as a limit: as $s\to0^+$,
  $y_e(s)\to+\infty$, $x_e(s)\to0$, and the evolute tangent tends to $x=0$,
  the clothoid normal encoding $D'(0)=0$. In the frame
  $[\boldsymbol{O};[0,-1]^{\mathsf{T}},[1,0]^{\mathsf{T}}]$, $\Gamma_1$ is a convex graph with
  left endpoint $b=-\infty$, so the auxiliary tangency results and the bounds
  in \Cref{tab:tangents} still apply; moreover,
  $x_e(s_{\pi/2})<x_e(s_{2\pi})$, making the subdivision the limiting Case A
  of \Cref{fig:plane-division}. Hence
  \cref{eq:total-tangent-bound,lem:global-tangency-count} gives
  $N(\boldsymbol{q})\le3$ also for $k_0=0$, completing the proof of
  \Cref{thm:main-prime}.

  \subsection{Effectiveness and Optimization of the Distance Algorithm}
  The distance algorithm uses the endpoints as its initial candidates, replacing the
  left candidate by an iteration from $s=0$ when $d'(0)<0$ and the right
  candidate by an iteration from $s=L$ when $d'(L)>0$. Since $D'=d\,d'$ for
  $d>0$, these tests may equivalently be written using $D'$; if $d=0$, the
  global minimum is already known. 
  The midpoint branch has two distinct roles: it provides an additional candidate when neither endpoint test activates a search, and it serves as a numerical fallback when at least one endpoint search is activated but none of the activated iterations succeeds. Our analysis concerns only the first role.

  By \Cref{thm:main}, there are at most
  three stationary points, while stationary inflection points do not change
  the sign of $D'$ and local extrema alternate. Thus every interior local
  minimum activates an endpoint search unless it is the middle extremum of
  a max--min--max configuration. In that remaining case, $D'(0)\ge0$ and
  $D'(L)\le0$, so either endpoint test activates a search and the midpoint provides the additional candidate. Hence the original
  candidate-selection logic covers every possible minimum configuration.

  We next prove that the remaining configuration cannot occur. Under the
  standard normalization,
  \begin{equation}
    D'(s_0)=-q_x,
    \qquad
    D'(s_{2\pi})=x(s_{2\pi})-q_x.
    \label{eq:endpoint-distance-signs}
  \end{equation}
  If the extended $2\pi$ segment has three stationary points, then
  $N(\boldsymbol{q})\le\widehat N(\boldsymbol{q})\le3$ from
  \cref{eq:total-tangent-bound,lem:global-tangency-count} forces
  $N(\boldsymbol{q})=\widehat N(\boldsymbol{q})=3$; hence
  \cref{eq:three-tangency-strip,eq:endpoint-distance-signs} give
  $D'(s_0)\le0$ and $D'(s_{2\pi})\ge0$.
  \begin{theorem}[Min--Max--Min Pattern]
    \label{min-max-min}
    Let $\boldsymbol{p}(s)$ denote a proper no-inflection clothoid segment with
    tangent-angle variation at most $2\pi$. The conditions
    $D'(0)\ge0$ and $D'(L)\le0$ preclude an interior local minimum.
    Consequently, if three stationary points are all local extrema, their
    order is min--max--min.
  \end{theorem}
  \begin{proof}
    Suppose that an interior local minimum $m$ exists under the stated
    endpoint signs. It is a simple zero of $D'$, and continuity gives nearest
    stationary points $\ell<m<r$ on its two sides. Normalize the segment and
    extend it only terminally to tangent-angle variation $2\pi$, so $s_0=0$.
    The three stationary points persist, and \Cref{thm:main} allows no others;
    hence the endpoint inequalities above give $D'(s_0)\le0$. Together with
    $D'(0)\ge0$, this yields $D'(s_0)=0$, so $\ell=s_0$ and $q_x=0$.

    Direct calculation from
    \cref{eq:clothoid_arc_parametrization,eq:evolute_arc_parametrization_with_k0},
    using $\theta$ as the integration variable and pairing opposite
    half-periods, gives
    \begin{equation}
      x(s_{2\pi})>x(s_0),
      \qquad
      x_e(s_\alpha)>x_e(s_0)=0
      \quad(0<\alpha\le2\pi),
      \label{eq:positive-x-displacements}
    \end{equation}
    with limiting values understood when $k_0=0$.

    If $r<s_{2\pi}$ is simple, it is a local maximum and the absence of any
    later stationary point gives $D'(s_{2\pi})<0$, contradicting the endpoint
    inequality. If $r=s_{2\pi}$, the query point lies on the two parallel
    endpoint normal lines, which are distinct by the first inequality in
    \cref{eq:positive-x-displacements}. The remaining possibility is a
    multiple zero $r<s_{2\pi}$. Then $D'(r)=D''(r)=0$, where
    $
    D''(r)=1+\kappa(r)
    (\boldsymbol{p}(r)-\boldsymbol{q})^{\mathsf{T}}\mathbf{N}(r),
    $
    so $\boldsymbol{q}=\boldsymbol{e}(r)$. The second inequality in
    \cref{eq:positive-x-displacements} gives $q_x>0$, contradicting $q_x=0$.
    Hence the endpoint conditions exclude an
    interior local minimum. A max--min--max configuration satisfies these
    conditions, so three local extrema must instead have min--max--min
    order.
  \end{proof}

  \Cref{min-max-min} therefore eliminates the midpoint initialization as an additional geometric candidate when both endpoint derivative tests are inactive. Indeed, under the condition that 
  $D'(0) \ge 0$ and $D'(L) \le 0$, 
  no interior local minimum exists, and hence the global minimum is atteined at one of the two endpoints. This conclusion does not concern numerical failure of an endpoint iteration after that iteration has been activated. Accordingly, the midpoint initialization may still be retained as a numerical fallback when the activated endpoint searches fail to produce a valid candidate. \Cref{alg:candidate-selection} makes this separation explicit.

  \section{Numerical Experiments}
  \label{sec:experiments}

  We compare the reduced-candidate(\emph{RC}) method with the original candidate-selection(\emph{ORI}) procedure 
  of \cite{frego2019point} on
  no-inflection segments with $k=1/2$, $k_0=\sqrt{\pi/2}$, and angle
  variations $n\pi/3$, $n=1,2,4,5$. 
  In the present experiments, no activated endpoint iteration failed; hence the numerical fallback was never invoked, and the RC implementation used here coincides with the safeguarded reduced-candidate procedure on all tested cases.
  Query points form a
  $1000\times1000$ grid over the segment bounding box enlarged by one unit on
  each side, and the tolerance is $\epsilon=2^{-23}$. Experiments use
  \textsf{Python~3.11.14}, \textsf{NumPy~2.3.5} with Intel MKL~2023.1, and \textsf{SciPy~1.16.3} in
  double precision on a single-threaded Intel Core i7-13700HX computer with
  32~Gigabytes of memory and Windows~11 Pro for Workstations. For the iteration-count
  experiment, each vectorized call processes the full query grid.
  \Cref{fig:iterations-maps} shows the pointwise iteration counts and savings
  on a shared scale.
  \begin{figure}[htbp]
    \centering
    \setlength{\tabcolsep}{4pt}
    \begin{tabular}{c@{\hspace{8pt}}cccc}
      & \footnotesize{$\Delta\theta = \pi/3$}
      & \footnotesize{$\Delta\theta = 2\pi/3$}
      & \footnotesize{$\Delta\theta = 4\pi/3$}
      & \footnotesize{$\Delta\theta = 5\pi/3$} \\[5pt]
      \rotatebox{90}{\footnotesize  the RC method}
      & \includegraphics[width=0.2\textwidth]{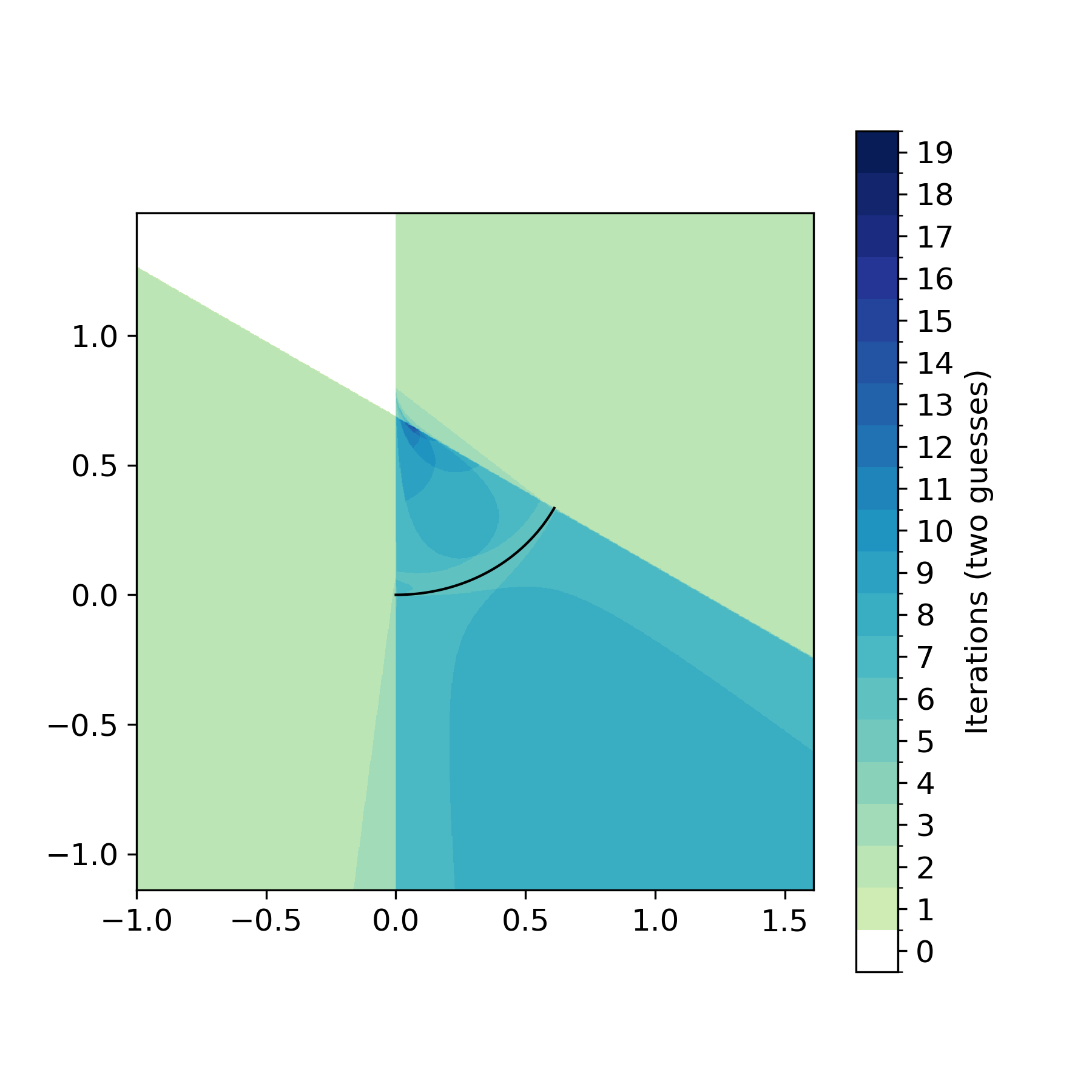}
      & \includegraphics[width=0.2\textwidth]{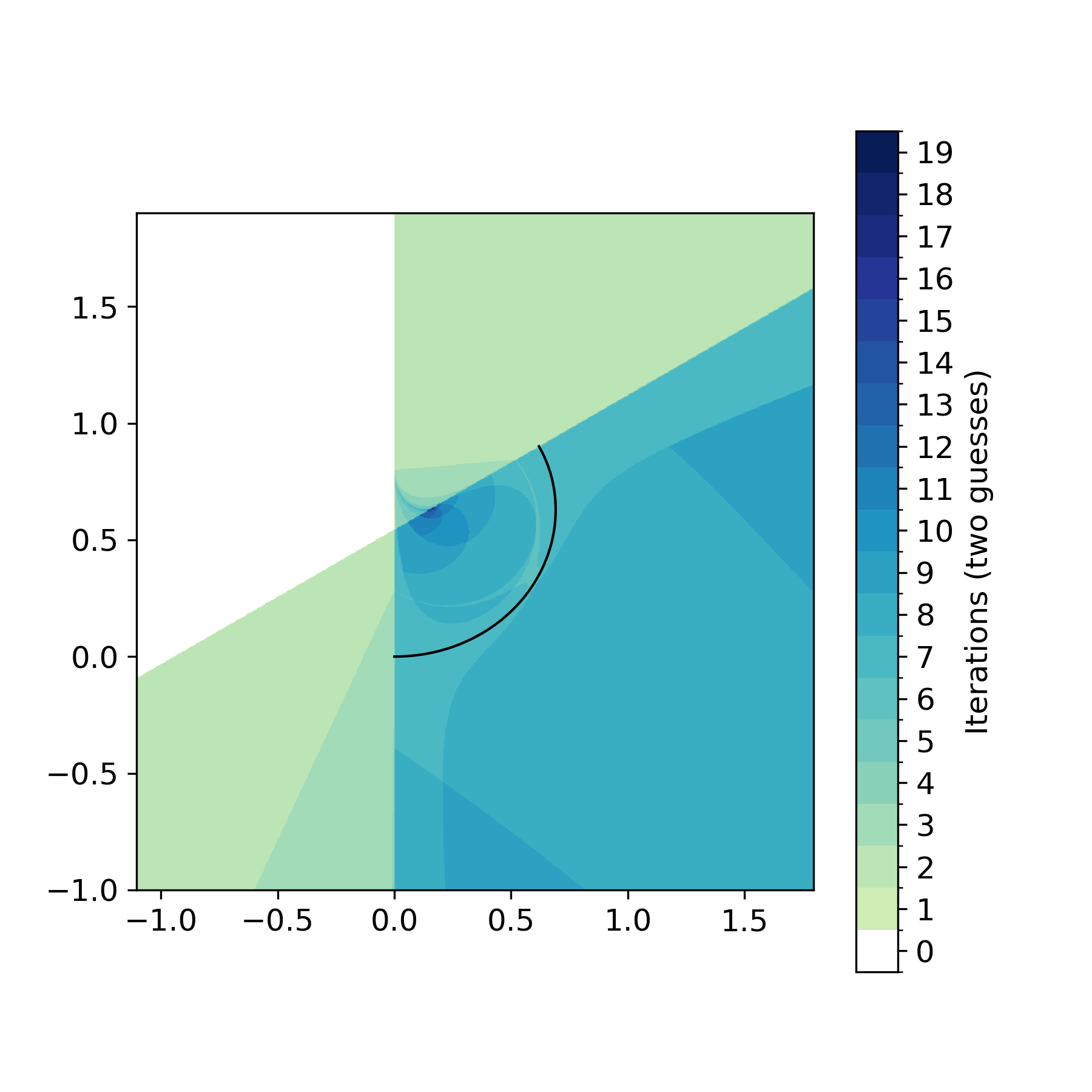}
      & \includegraphics[width=0.2\textwidth]{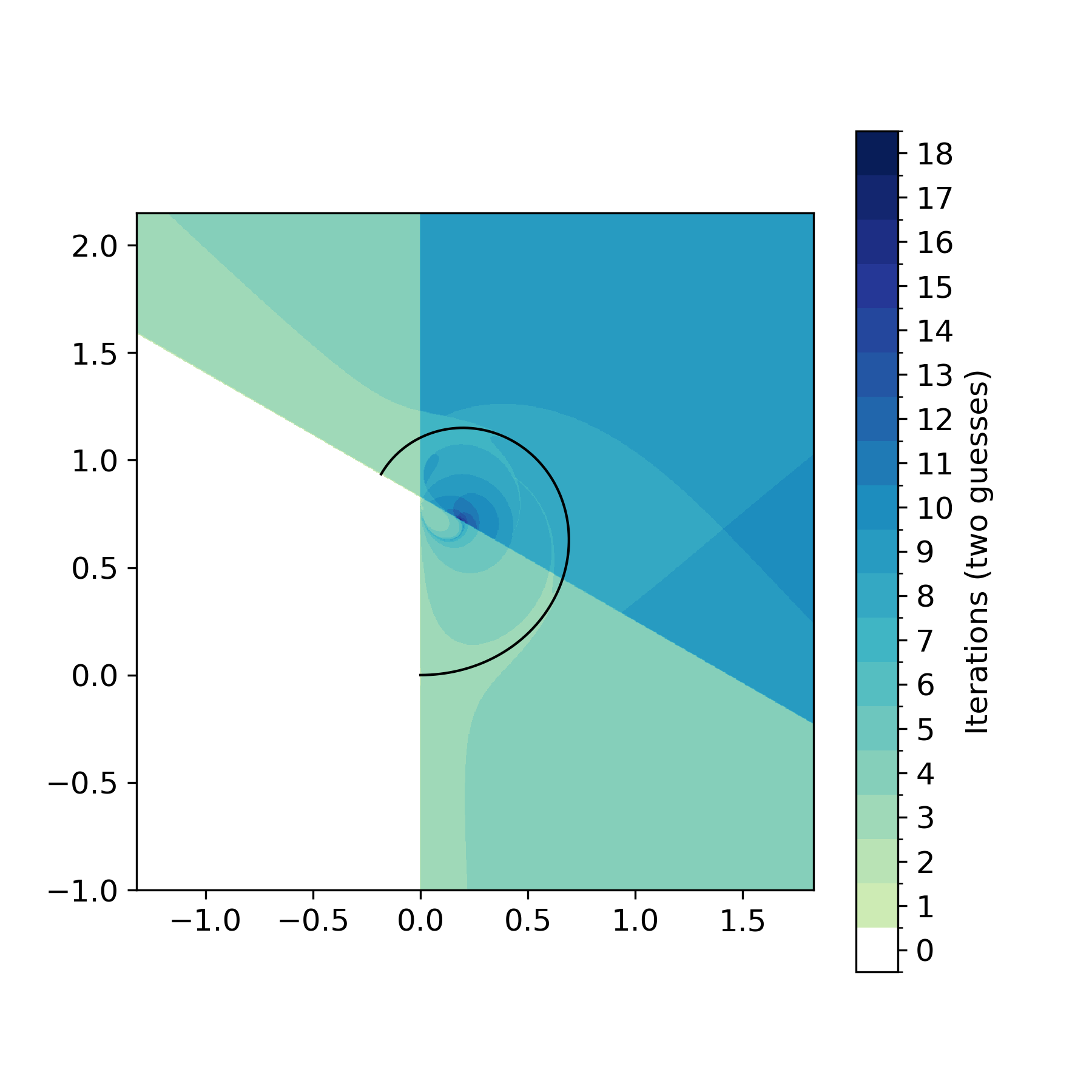}
      & \includegraphics[width=0.2\textwidth]{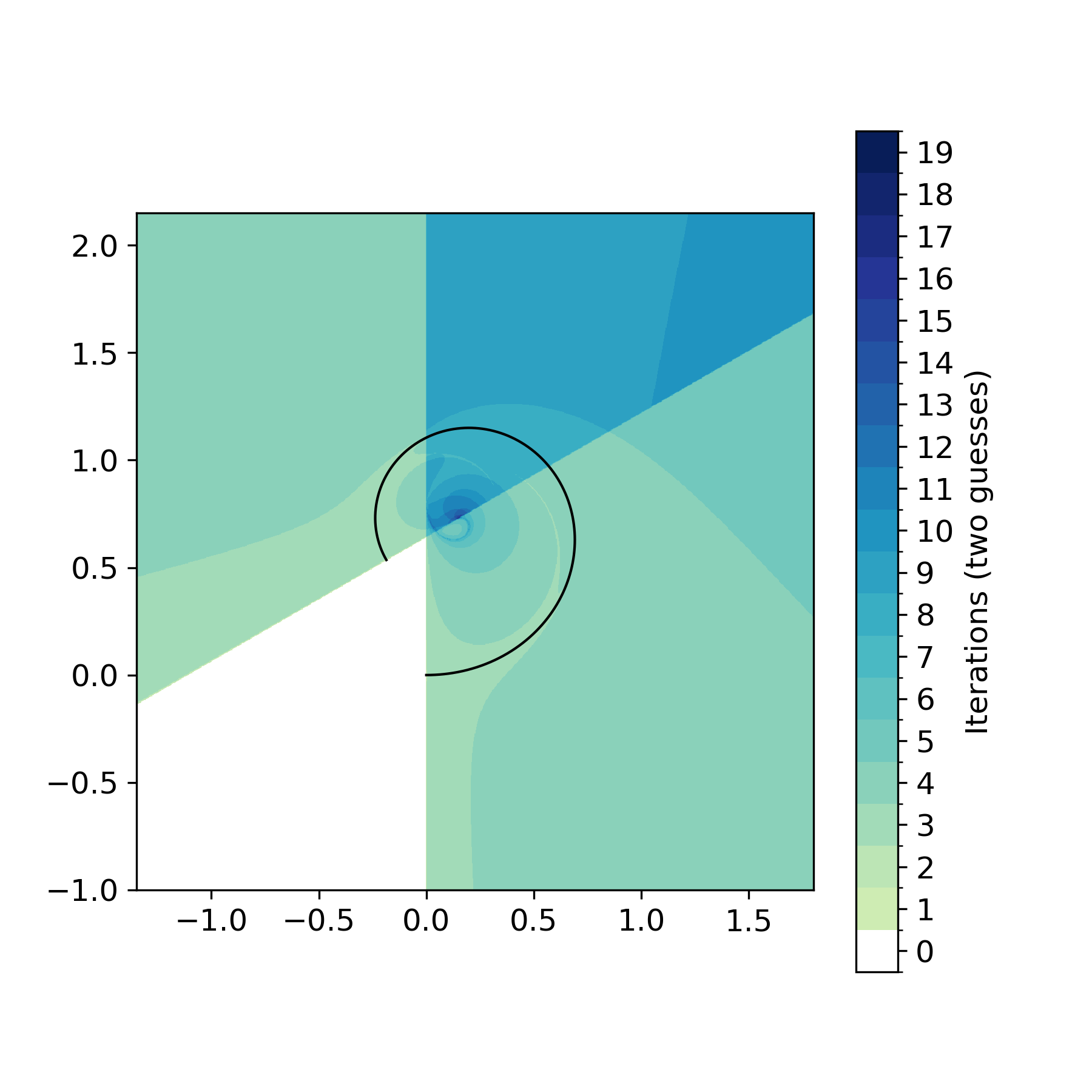} \\[5pt]
      \rotatebox{90}{\footnotesize the ORI method}
      & \includegraphics[width=0.2\textwidth]{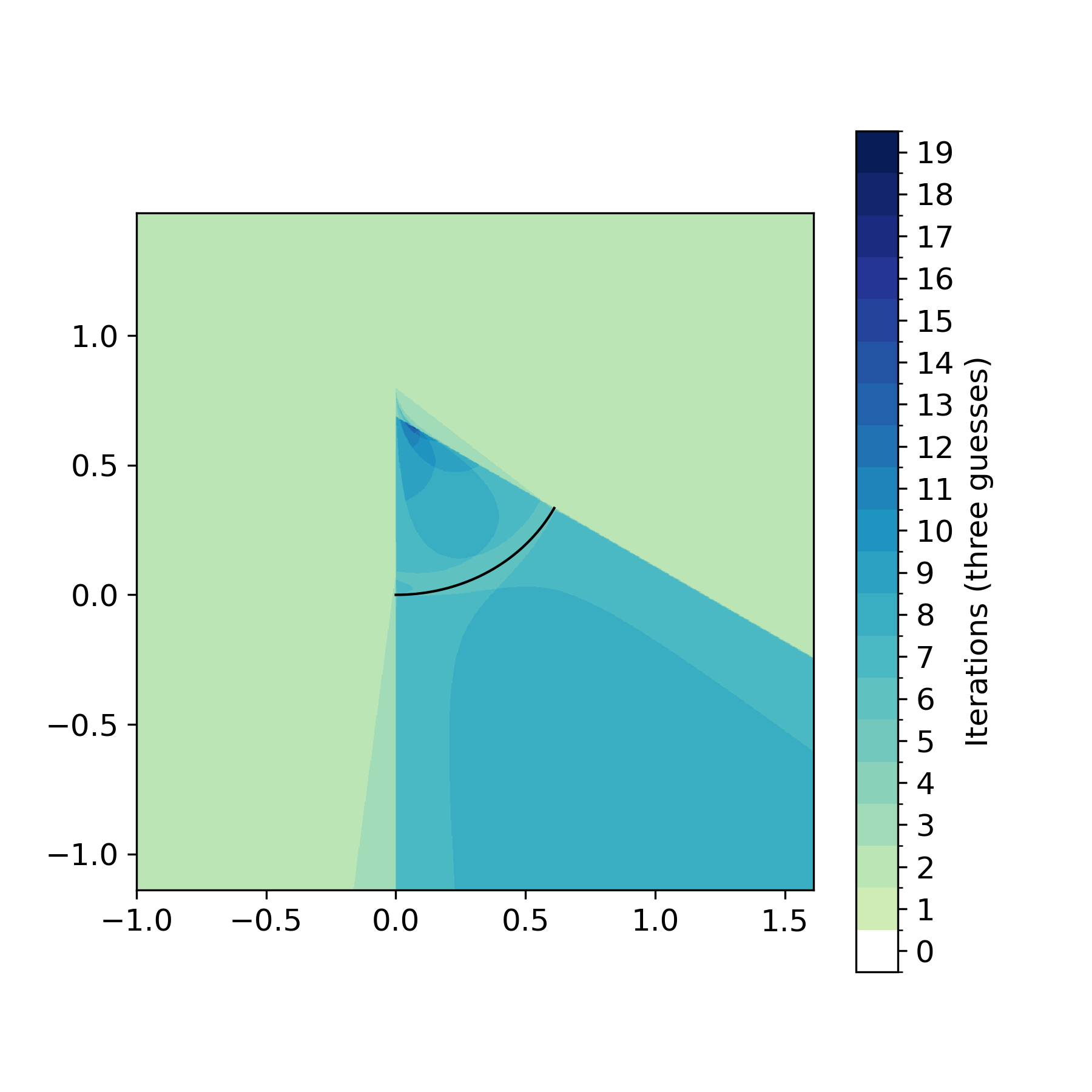}
      & \includegraphics[width=0.2\textwidth]{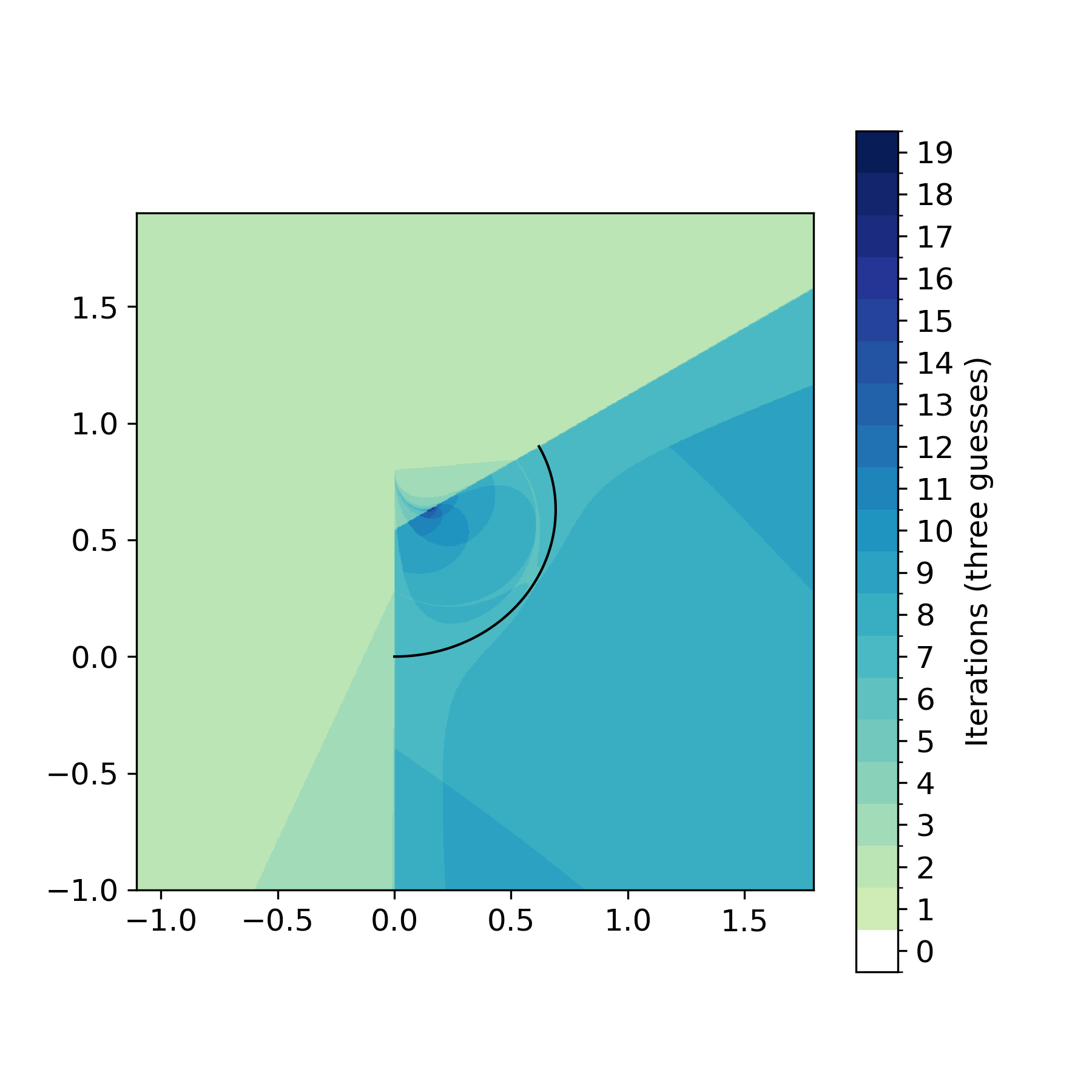}
      & \includegraphics[width=0.2\textwidth]{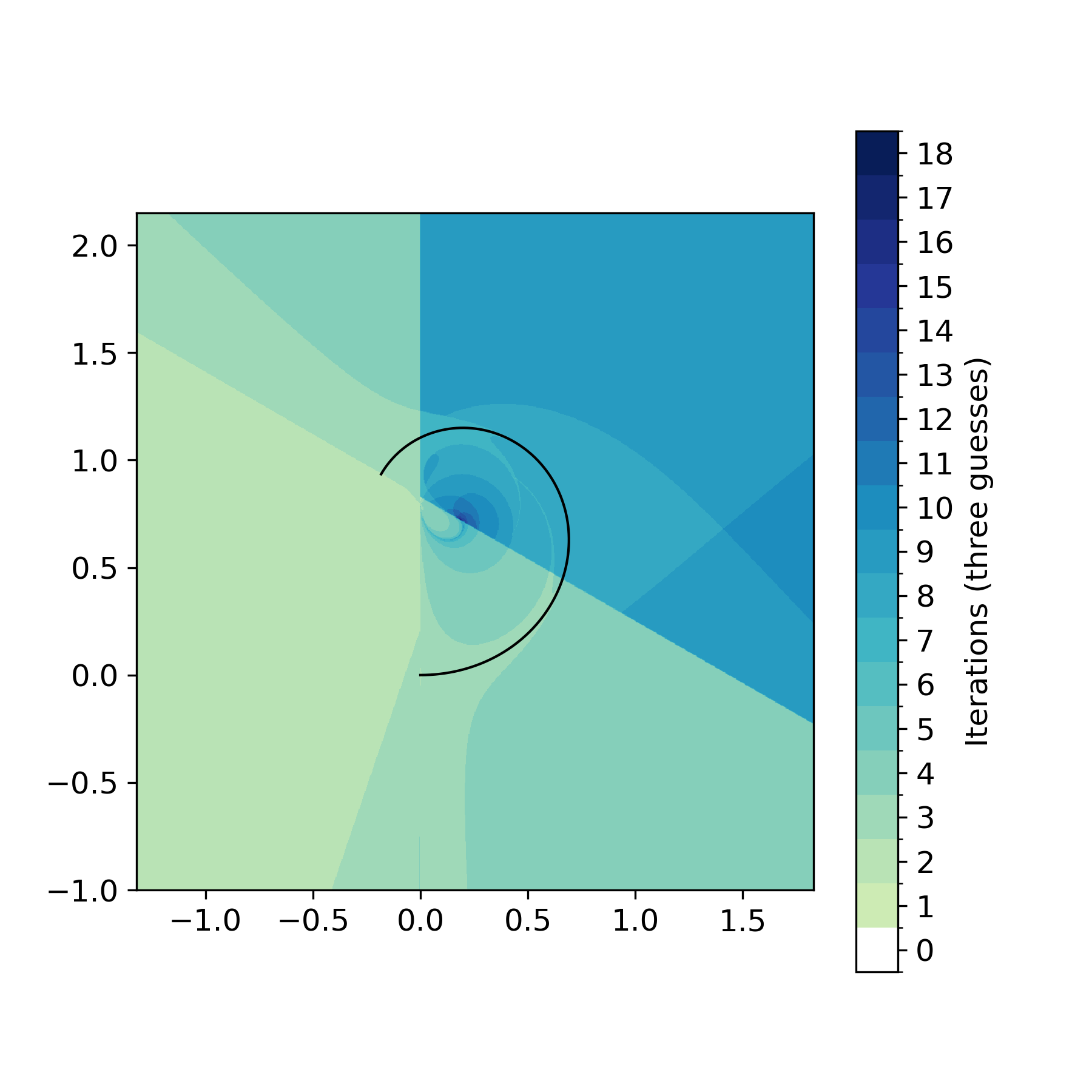}
      & \includegraphics[width=0.2\textwidth]{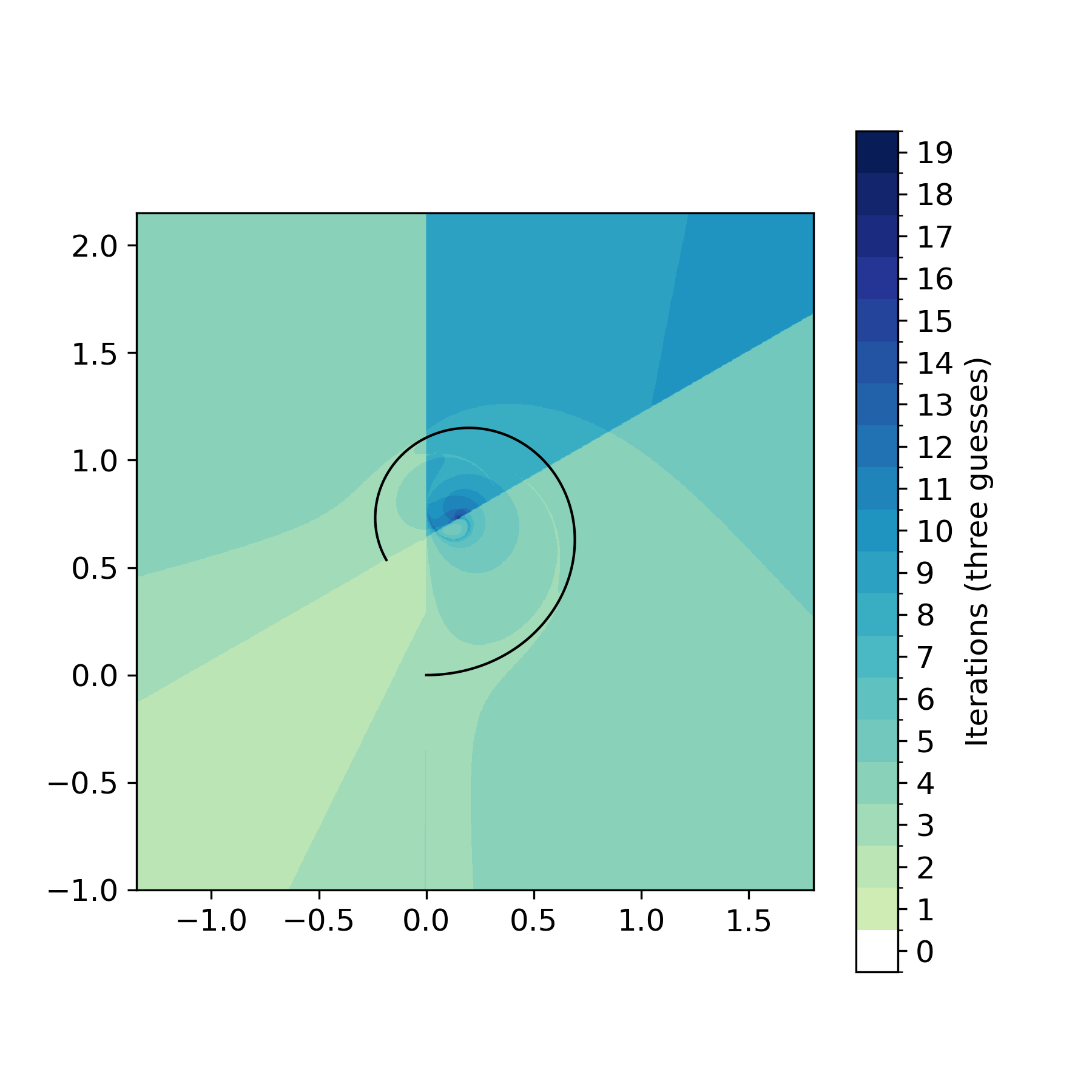} \\[5pt]
      \rotatebox{90}{\footnotesize iterations saved}
      & \includegraphics[width=0.2\textwidth]{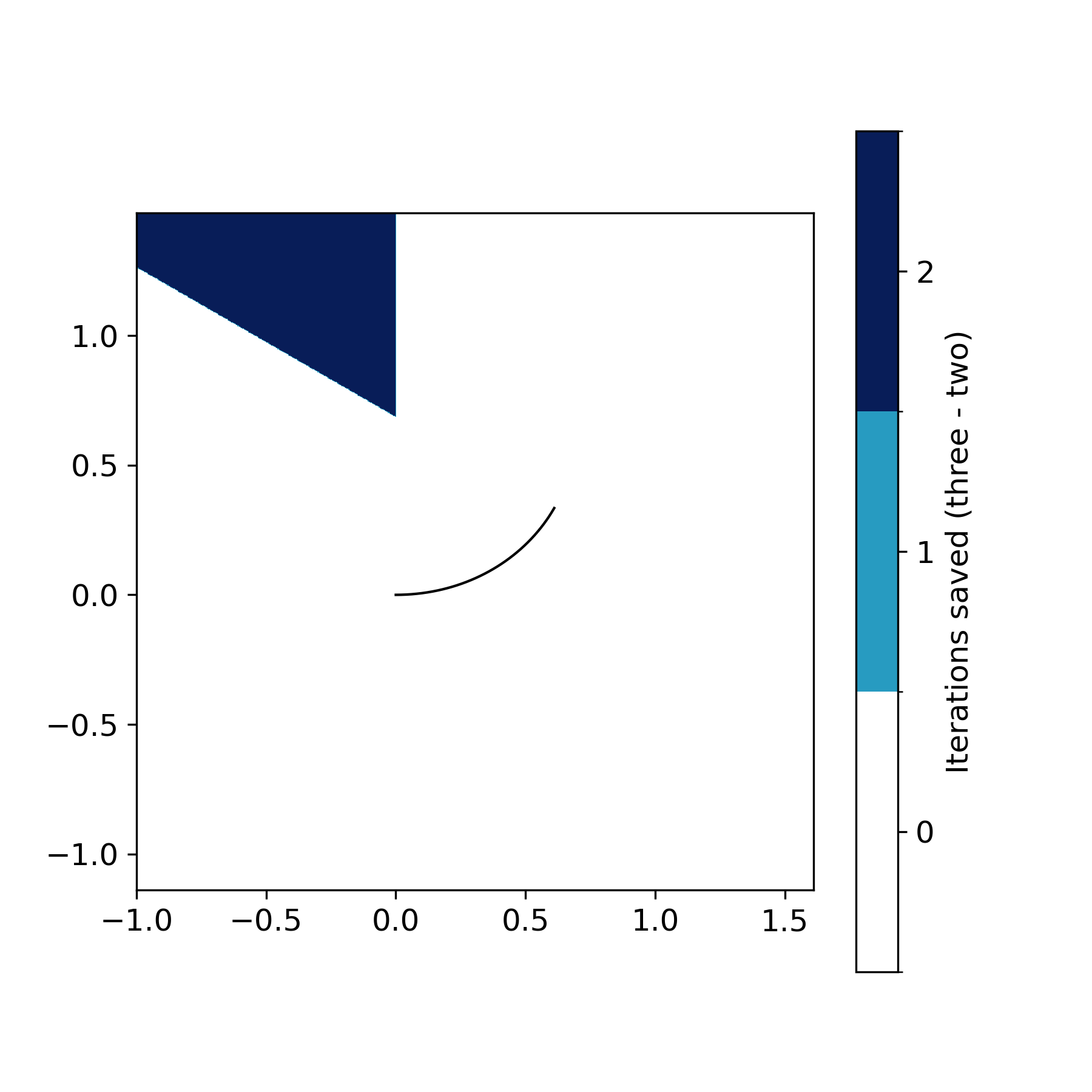}
      & \includegraphics[width=0.2\textwidth]{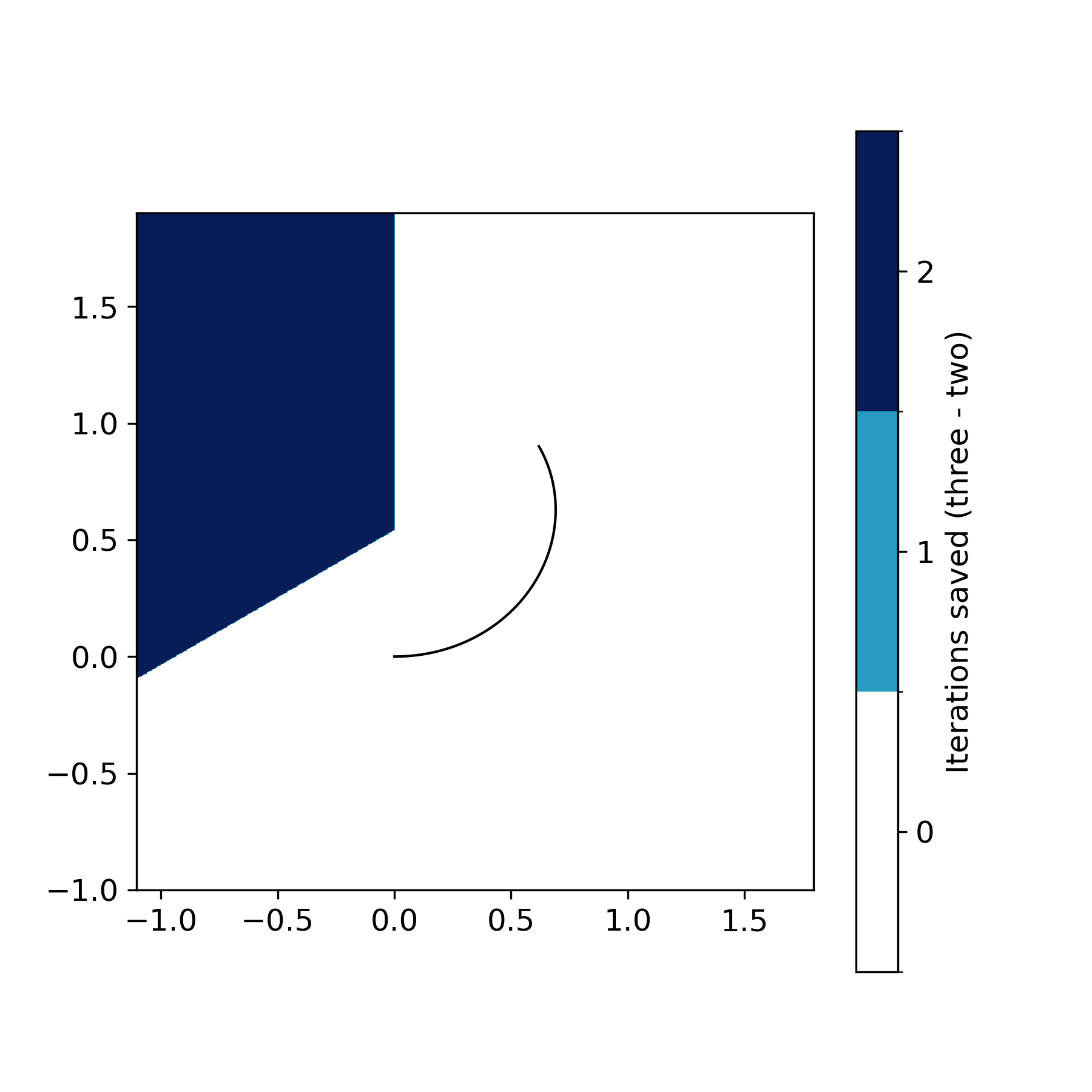}
      & \includegraphics[width=0.2\textwidth]{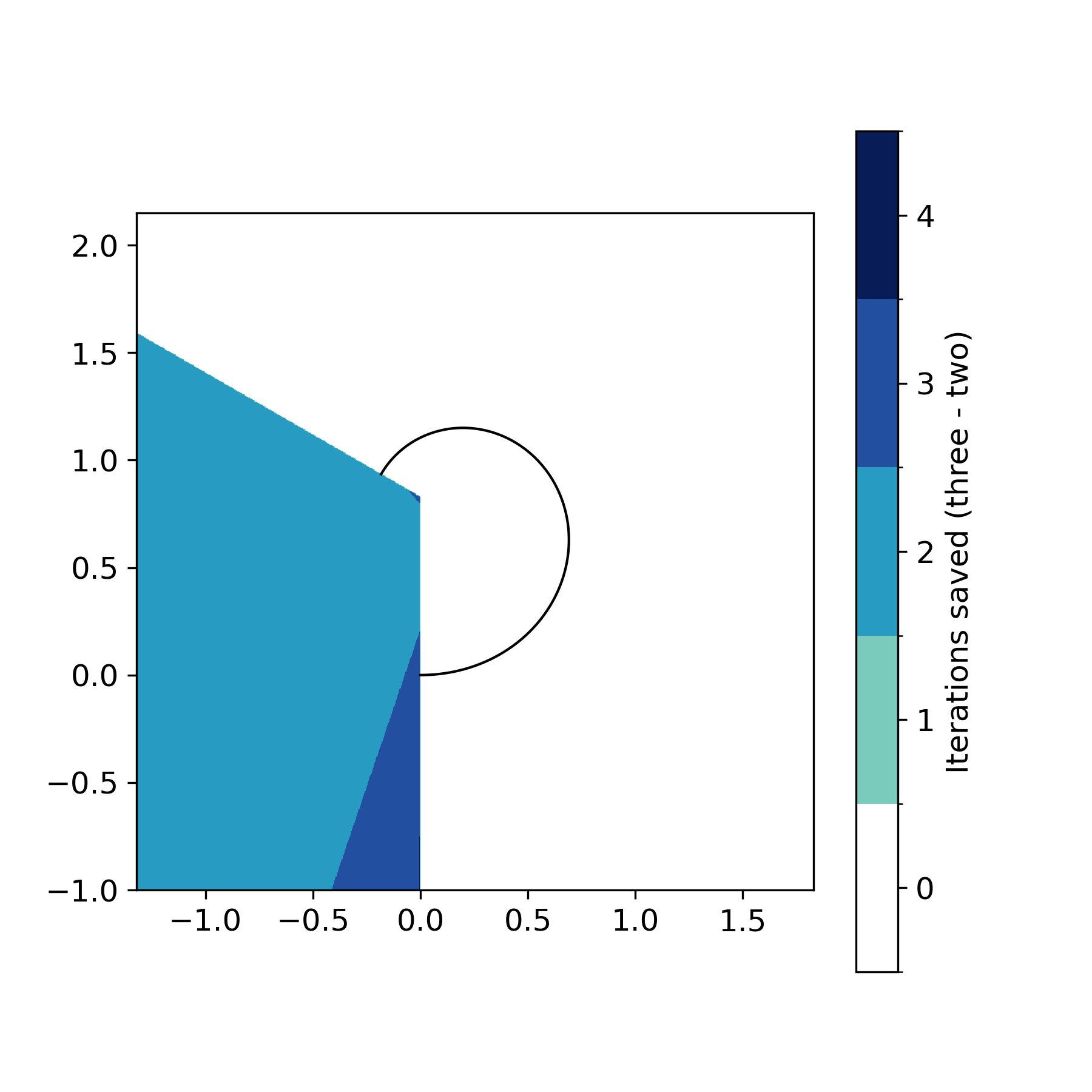}
      & \includegraphics[width=0.2\textwidth]{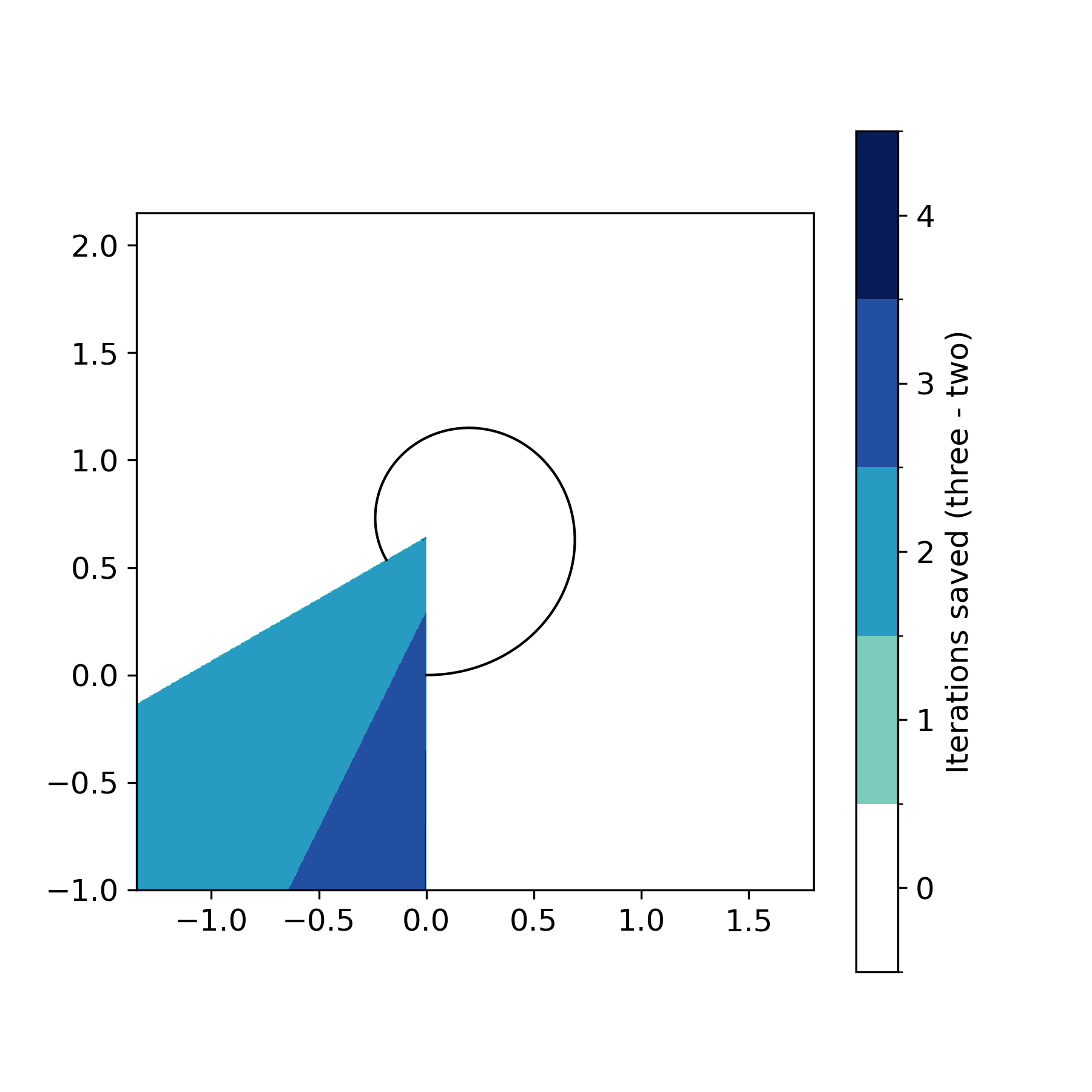}
    \end{tabular}
    \caption{Iteration maps for $k=1/2$, $k_0=\sqrt{\pi/2}$, and
      $\Delta\theta=n\pi/3$, $n=1,2,4,5$ (columns). The top and middle rows
      show the RC and ORI counts on a shared scale; the
      bottom row shows the iterations saved, with darker cells indicating
      larger savings. Each cell displays the average of its four corner
      samples, so the bottom row is the cellwise difference of the two rows
    above.}
    \label{fig:iterations-maps}
  \end{figure}

  The algorithms agree outside the marked region where neither endpoint search is activated, i.e., where $d'(0)\ge0$ and $d'(L)\le0$. Since
  $d'(s)=\cos(\theta(s)-\phi(s))$, with $\phi(s)$ the polar angle of
  $\boldsymbol{p}(s)-\boldsymbol{q}$, this region is geometrically the
  intersection of the half-planes behind the start point and beyond the end
  point, bounded by the endpoint normals: it
  is a narrow wedge for $\Delta\theta=\pi/3$, covers nearly the left half-plane
  for $2\pi/3$, and extends farther downward for $4\pi/3$ and $5\pi/3$, with
  the largest savings along the initial normal. This region contains
  7.29\%--29.40\% of the grid; skipping the midpoint guess saves at most four
  iterations per marked point and reduces the full-grid iteration count by
  3.80\%--12.29\%. For timing, we use only these marked points and report the mean compute time
  of $10^4$ alternating vectorized calls per algorithm and segment. The algorithm for
  RC method saves
  60.9\%--70.3\% of the evaluation time, while \Cref{min-max-min} guarantees
  that this omission removes no geometric minimum-distance candidate.
  \begin{table}[htbp]
    \scriptsize
    \caption{Iteration and timing comparison for $k=1/2$,
      $k_0=\sqrt{\pi/2}$, and $\Delta\theta=n\pi/3$, $n=1,2,4,5$.
      Iteration statistics use the
      $1000\times1000$ grid; timings use the marked points and report the
    mean of $10^4$ alternating vectorized evaluations.}
    \label{tab:performance}
    \begin{center}
      \setlength{\tabcolsep}{2.5pt}
      \begin{tabular}{|c|c|c|c|c|c|c|c|c|} \hline
        & marked & \multicolumn{4}{c|}{stationary-point iterations}
        & \multicolumn{3}{c|}{compute time on \textbf{marked points}} \\
        \cline{3-9}
        $\Delta\theta$ & \shortstack{points\\(\%)}
        & \shortstack{RC\\mean} & \shortstack{ORI\\mean}
        & \shortstack{max iter.\\saved} & \shortstack{total iter.\\reduction (\%)}
        & \shortstack{RC\\(ms)} & \shortstack{ORI\\(ms)}
        & \shortstack{time saved\\(\%)} \\ \hline
        $\pi/3$  & 7.29  & 3.691 & 3.837 & 2 & 3.80  & 7.00  & 23.56  & 70.3 \\
        $2\pi/3$ & 22.02 & 4.235 & 4.676 & 2 & 9.42  & 26.19 & 88.11  & 70.3 \\
        $4\pi/3$ & 29.40 & 4.379 & 4.992 & 4 & 12.29 & 63.85 & 174.24 & 63.4 \\
        $5\pi/3$ & 17.03 & 4.224 & 4.607 & 4 & 8.31  & 36.28 & 92.67  & 60.9 \\ \hline
      \end{tabular}
    \end{center}
  \end{table}

  \section{Conclusions}
  \label{sec:conclusions}
  In this manuscript, we revisited the point-clothoid distance algorithm of Frego and Bertolazzi. The one-local-minimum characterization in the underlying analysis is not exhaustive: a proper no-inflection clothoid segment with tangent-angle variation at most $2\pi$ may contain two local minima. To determine whether the candidate-selection strategy keeps complete, we recast stationary points of the squared-distance function as tangencies from the query point to the clothoid evolute. Decomposing the evolute into four \(\pi/2\) segments and exploiting their strict convexity yields a global tangency bound, including the zero-initial-curvature case.

  The resulting classification shows that there are at most three stationary points and that three local extrema can occur only in min--max--min order. Consequently, the endpoint-based candidate-selection logic fully accounts for all possible geometric minimum candidates. When neither endpoint test is active, no interior local minimum exists, and the midpoint need not be searched as an additional candidate. This does not eliminate its separate role as a numerical fallback. The resulting procedure therefore simplifies the candidate-selection logic while preserving the numerical safeguarding mechanism of the original method. The numerical experiments quantify the corresponding savings in iteration count and evaluation time.

  \appendix
  \section{Proofs of the auxiliary tangency bounds}
  \label{sec:supp-auxiliary-tangency-proofs}

  This section contains the proofs of \Cref{prop:y<0,prop:y>f(x),cor:convex-extension,prop:all-points} in the main manuscript, together with an auxiliary supporting-line lemma used in the proof of \Cref{prop:all-points}.

  \subsection*{Proof of \Cref{prop:y<0}}

  \begin{proof}
    Since $f'>0$ and $f(x)\to0$ as $x\to b^+$, we have $f>0$ on
    $(b,a)$. Suppose two interior tangents at $x_1<x_2$ pass through
    $\boldsymbol{q}$. Then
    \begin{equation}
      \label{eq:supp-system-of-equations-f}
      q_y=f(x_i)+f'(x_i)(q_x-x_i),\qquad i=1,2.
    \end{equation}
    Eliminating $q_x$ gives
    \begin{equation}
      \label{eq:supp-explicit-expression-of-y}
      q_y=\frac{f'(x_1)f'(x_2)}{f'(x_2)-f'(x_1)}
      [\phi(x_2)-\phi(x_1)],
    \end{equation}
    where
    \begin{equation}
      \label{eq:supp-expression-of-phi}
      \phi(x)=x-\frac{f(x)}{f'(x)},
      \qquad
      \phi'(x)=\frac{f(x)f''(x)}{[f'(x)]^2}>0.
    \end{equation}
    Since $f''>0$, $f'(x_2)>f'(x_1)$, so
    \cref{eq:supp-explicit-expression-of-y,eq:supp-expression-of-phi} imply
    $q_y>0$, a contradiction.

    The endpoint cases give the same conclusion. If $b$ is finite, its tangent
    is $y=0$. If the right endpoint is finite, its limiting tangent is $x=a$;
    its intersection with the tangent at $x_1<a$ has ordinate
    $f(x_1)+f'(x_1)(a-x_1)>0$. Thus a point with $q_y<0$ lies on at most one
    tangent.

    Finally, if $q_x>a$, every interior tangent at $x_0$ has ordinate
    $f(x_0)+f'(x_0)(q_x-x_0)>0$ at $q_x$. Neither endpoint tangent can contain
    a point with $q_x>a$ and $q_y<0$, so no tangent passes through
    $\boldsymbol{q}$.
  \end{proof}

  \subsection*{Proof of \Cref{prop:y>f(x)}}
  \begin{proof}
    A nonvertical tangent at $x_0\in[b,a)$ passing through
    $\boldsymbol{q}$ would satisfy
    \begin{equation}
      \label{eq:supp-tangent-equation}
      q_y=f(x_0)+f'(x_0)(q_x-x_0).
    \end{equation}
    Convexity gives
    $f(q_x)\ge f(x_0)+f'(x_0)(q_x-x_0)=q_y$, contradicting
    $q_y>f(q_x)$. The limiting tangent at the right endpoint is the vertical
    line $x=a$, which cannot contain a point with $q_x<a$.
  \end{proof}

  \subsection*{Proof of \Cref{cor:convex-extension}}

  \begin{proof}
    The conditions $f(b)=f'(b)=0$ make the constant extension convex. The
    supporting-line argument in the proof of \Cref{prop:y>f(x)} then applies
    directly to $F$.
  \end{proof}

  \subsection*{An Auxiliary Lemma}

  \begin{lemma}[Auxiliary Lemma for \Cref{prop:all-points}]
    \label{lem:aux-lemma}
    Let $C\subset\mathbb{R}^2$ be a nonempty closed convex set and
    $\boldsymbol{q}\notin C$. At most two supporting lines of $C$ pass through
    $\boldsymbol{q}$. Consequently, where tangent lines are defined, at most two
    can pass through $\boldsymbol{q}$.
  \end{lemma}

  \begin{proof}
    Strict separation gives a vector $\boldsymbol{v}$ and $c>0$ such that
    \[
      \boldsymbol{v}^T(\boldsymbol{x}-\boldsymbol{q})\ge c
      \qquad\text{for all }\boldsymbol{x}\in C.
    \]
    Hence the radial projection
    \[
      R(\boldsymbol{x})=
      \frac{\boldsymbol{x}-\boldsymbol{q}}
      {\|\boldsymbol{x}-\boldsymbol{q}\|}
    \]
    maps $C$ into an open semicircle of the unit circle. The map $R$ is
    continuous on $C$, and $C$ is connected because it is convex; hence
    $R(C)$ is connected. Every connected subset of an open semicircle is an
    arc, possibly degenerate, whose closure has at most two endpoints.

    A supporting line through $\boldsymbol{q}$ must correspond to an endpoint
    direction of this arc; an interior direction would place points of $C$ on
    both sides of the line. Since a direction and $\boldsymbol{q}$ determine a
    unique line, at most two supporting lines pass through $\boldsymbol{q}$.
    Every tangent line of a convex set is a supporting line, which proves the
    final assertion.
  \end{proof}

  \subsection*{Proof of \Cref{prop:all-points}}

  \begin{proof}
    Let $F=f$ when $b=-\infty$, and otherwise let $F$ be the extension in
    \Cref{eq:F(x)}. The closure of the epigraph of $F$ is convex, and every
    tangent line considered here is one of its supporting lines. If
    $\boldsymbol{q}$ lies outside this epigraph, the above \Cref{lem:aux-lemma} gives
    at most two such lines; an interior point lies on none. At a boundary point
    the supporting line is unique: it is the graph tangent, the horizontal
    tangent along the constant extension, or the limiting vertical tangent at
    $x=a$. This proves the general bound.

    For the special case, an interior tangent at $x_0\in(b,a)$ meets the
    $x$-axis at
    \begin{equation}
      \label{eq:supp-zero-of-tangent}
      x=\phi(x_0)=x_0-\frac{f(x_0)}{f'(x_0)}.
    \end{equation}
    Since $f(x_0)>0$, we have $\phi(x_0)<x_0<a$. If $b$ is finite, strict
    increase of $f'$ gives $f(x_0)<f'(x_0)(x_0-b)$, hence
    $\phi(x_0)>b$. Thus every interior tangent meets the $x$-axis in $(b,a)$.
    Outside this interval, the only possible tangent is the horizontal endpoint
    tangent $y=0$ when $b> -\infty$, proving the sharper bound.
  \end{proof}

  \section{Exhaustive enumeration of the tangency bounds}
  \label{sec:supp-tangency-enumeration}

  This section organizes the exhaustive verification underlying
  \Cref{lem:global-tangency-count,thm:main-prime}. For every set in the plane
  subdivisions of \Cref{fig:plane-division}, the four segment-wise bounds
  $u_i(\boldsymbol{q})$ from \Cref{tab:tangents} are combined with the junction
  correction $\delta=\sum_{j=1}^3\delta_j(\boldsymbol{q})$ to give
  $\widehat N=\sum_{i=1}^4u_i-\delta$, an upper bound for the exact tangency
  count $N(\boldsymbol{q})$.

  \Cref{tab:tangents_all_case} gives the row-by-row verification. Each row
  identifies an open region, relative boundary interior, or intersection point
  and records the applicable case, the four component bounds, the total junction
  correction, and the resulting value of $\widehat N$.
  For example, $\Omega_{15}$ has $(u_1,u_2,u_3,u_4)=(2,0,0,1)$ and
  $\delta=0$, whereas $\Omega_{15}\cap\Omega_{16}$ has $(2,1,0,1)$ and
  $\delta=1$; both give $\widehat N=3$.

  The listed sets cover the whole plane. Primed regions occur only in their
  corresponding configurations, and the two entries for
  $\boldsymbol{e}(s_{\pi/2})$ distinguish Case A from Cases B and C. Equality
  in either defining comparison merely collapses adjacent sets, which remain
  covered by the same non-strict boundary bounds.

  For comparison, \Cref{fig:stationary-heatmaps} shows the exact stationary-point
  counts obtained on sampled planar grids for the three configurations. These
  counts are consistent with the analytic upper bounds enumerated below.

  \begin{figure}[htbp]
    \centering
    \begin{subfigure}{0.32\textwidth}
      \includegraphics[width=\textwidth]{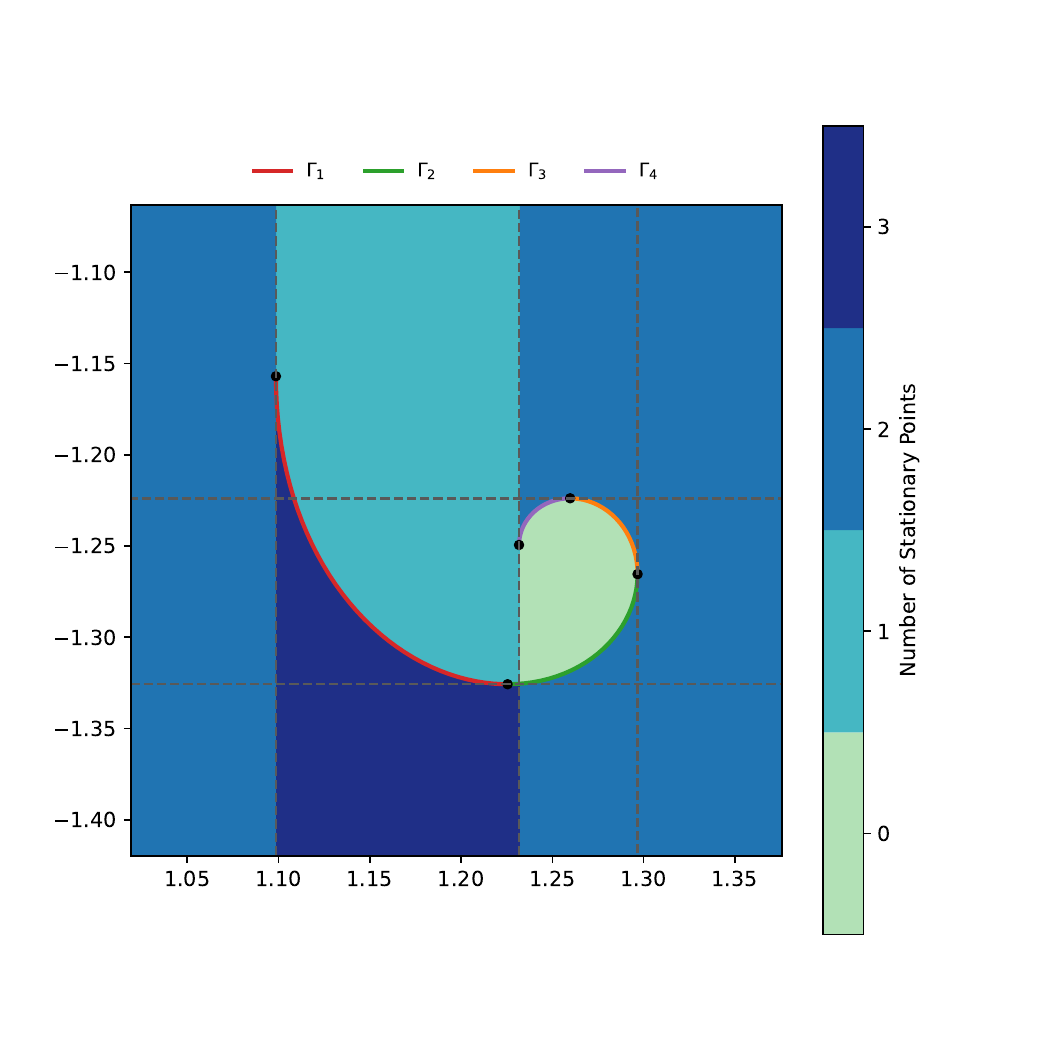}
      \caption{Case A}
    \end{subfigure}\hfill
    \begin{subfigure}{0.32\textwidth}
      \includegraphics[width=\textwidth]{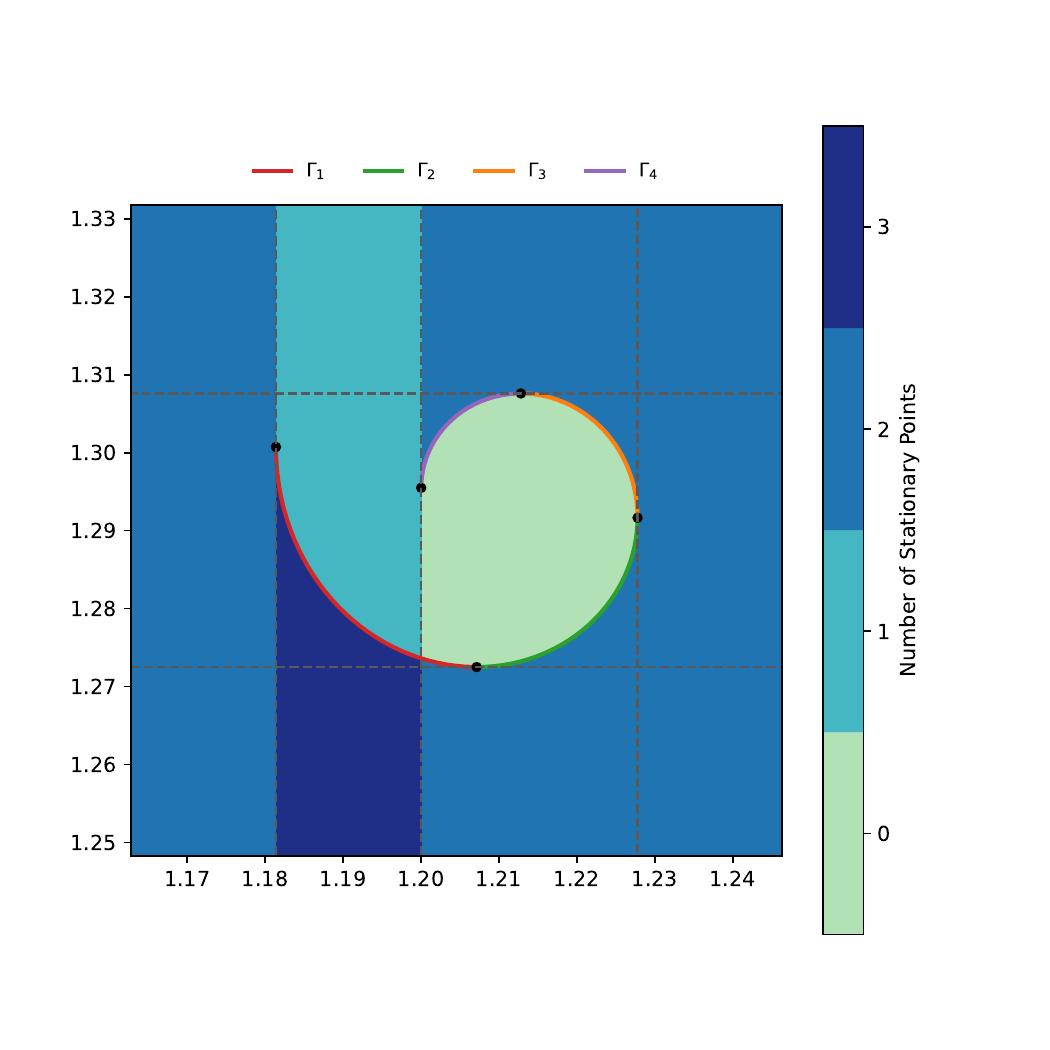}
      \caption{Case B}
    \end{subfigure}\hfill
    \begin{subfigure}{0.32\textwidth}
      \includegraphics[width=\textwidth]{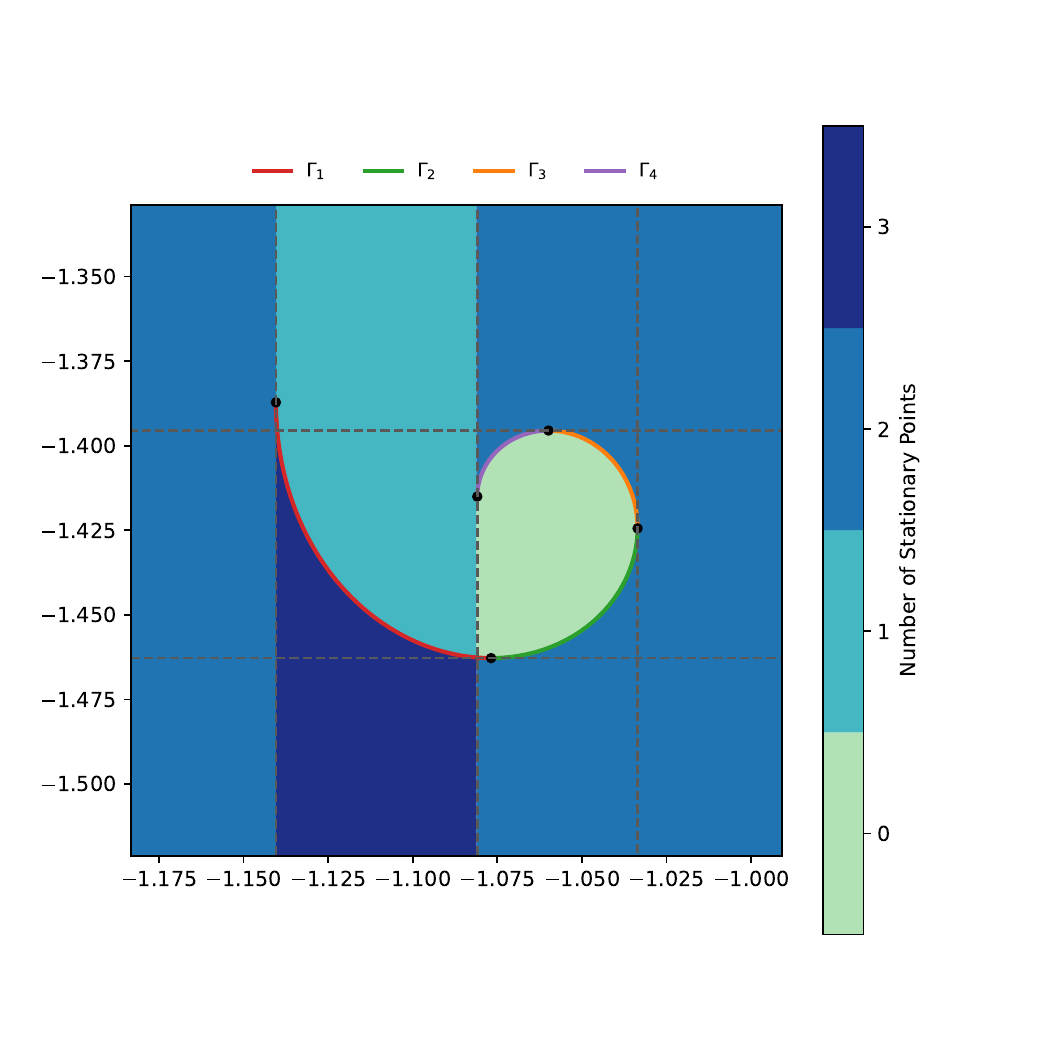}
      \caption{Case C}
    \end{subfigure}
    \caption{Number of stationary points of the squared-distance function at
      sampled points of the plane, equivalently, the number of evolute tangency
      parameters. The sampled counts do not exceed the analytic upper bounds in
    \Cref{tab:tangents_all_case}.}
    \label{fig:stationary-heatmaps}
  \end{figure}

  \footnotesize
  \setlength{\tabcolsep}{3.5pt}
  \begin{longtable}{|c|c|c|c|c|c|}
    \caption{Definition of closed region $\Omega_{i}$. The four-tuple of segmentwise regions defines a certain region $\Omega_{i}$ in the plane, where $A^{(i)}_j$ denote the region of $A_j$ for segment $\Gamma_i$ defined in Table~1. Note that for the same $\Gamma_{i}$, $A_{4}$ and $A_{5}$ share equivalent status, so $A_4^{(i)}$ denote both $A_{4}$ and $A_{5}$ in the table. }
    \label{tab:omega_tuples}\\
    \hline
    \textbf{Case} & \textbf{Region} & $\boldsymbol{(A^{(1)},A^{(2)},A^{(3)},A^{(4)})}$ & \textbf{Case} & \textbf{Region} & $\boldsymbol{(A^{(1)},A^{(2)},A^{(3)},A^{(4)})}$ \\ \hline
    \endfirsthead
    \multicolumn{6}{c}{\tablename~\thetable\ (continued)}\\
    \hline
    \textbf{Case} & \textbf{Region} & $\boldsymbol{(A^{(1)}_{j_1},A^{(2)}_{j_2},A^{(3)}_{j_3},A^{(4)}_{j_4})}$ & \textbf{Case} & \textbf{Region} & $\boldsymbol{(A^{(1)}_{j_1},A^{(2)}_{j_2},A^{(3)}_{j_3},A^{(4)}_{j_4})}$ \\ \hline
    \endhead
    \hline
    \multicolumn{6}{r}{Continued on the next page}\\
    \endfoot
    \hline
    \endlastfoot
    A--C & $\Omega_{1}$   & $(A^{(1)}_{1},A^{(2)}_{1},A^{(3)}_{1},A^{(4)}_{1})$ & A--C & $\Omega_{11}$  & $(A^{(1)}_{1},A^{(2)}_{3},A^{(3)}_{4},A^{(4)}_{1})$ \\
    A--C & $\Omega_{2}$   & $(A^{(1)}_{1},A^{(2)}_{1},A^{(3)}_{1},A^{(4)}_{4})$ & A--C & $\Omega_{12}$  & $(A^{(1)}_{4},A^{(2)}_{1},A^{(3)}_{1},A^{(4)}_{4})$ \\
    A--C & $\Omega_{3}$   & $(A^{(1)}_{1},A^{(2)}_{1},A^{(3)}_{4},A^{(4)}_{3})$ & A--C & $\Omega_{13}$  & $(A^{(1)}_{4},A^{(2)}_{1},A^{(3)}_{4},A^{(4)}_{3})$ \\
    A--C & $\Omega_{4}$   & $(A^{(1)}_{1},A^{(2)}_{1},A^{(3)}_{1},A^{(4)}_{2})$ & A--C & $\Omega_{14}$  & $(A^{(1)}_{3},A^{(2)}_{4},A^{(3)}_{1},A^{(4)}_{4})$ \\
    A--C & $\Omega_{5}$ & $(A^{(1)}_{1},A^{(2)}_{1},A^{(3)}_{2},A^{(4)}_{1})$ & A--C & $\Omega_{15}$  & $(A^{(1)}_{2},A^{(2)}_{1},A^{(3)}_{1},A^{(4)}_{4})$ \\
    A--C & $\Omega_{6}$   & $(A^{(1)}_{1},A^{(2)}_{2},A^{(3)}_{1},A^{(4)}_{1})$ & A--C & $\Omega_{16}$  & $(A^{(1)}_{4},A^{(2)}_{4},A^{(3)}_{1},A^{(4)}_{4})$ \\
    A--C & $\Omega_{7}$   & $(A^{(1)}_{1},A^{(2)}_{1},A^{(3)}_{4},A^{(4)}_{4})$ & A,C  & $\Omega_{15}'$ & $(A^{(1)}_{2},A^{(2)}_{1},A^{(3)}_{4},A^{(4)}_{3})$ \\
    A--C & $\Omega_{8}$   & $(A^{(1)}_{1},A^{(2)}_{4},A^{(3)}_{4},A^{(4)}_{1})$ & A    & $\Omega_{6}'$  & $(A^{(1)}_{1},A^{(2)}_{2},A^{(3)}_{1},A^{(4)}_{4})$ \\
    A--C & $\Omega_{9}$   & $(A^{(1)}_{1},A^{(2)}_{4},A^{(3)}_{3},A^{(4)}_{4})$ & B,C  & $\Omega_{15}''$ & $(A^{(1)}_{2},A^{(2)}_{1},A^{(3)}_{1},A^{(4)}_{1})$ \\
    A--C & $\Omega_{10}$  & $(A^{(1)}_{4},A^{(2)}_{4},A^{(3)}_{1},A^{(4)}_{1})$ &      &                & \\ \hline
  \end{longtable}
  \normalsize

  \footnotesize
  \setlength{\tabcolsep}{3pt}
  \begin{longtable}{|c|c|p{0.36\textwidth}|c|c|c|c|c|c|}
    \caption{Verification worksheet for the upper bound
      $\widehat N(\boldsymbol{q})=\sum_{i=1}^{4}u_i(\boldsymbol{q})-\delta$
      on the number $N(\boldsymbol{q})$ of evolute tangency parameters. Each row
      identifies a region, boundary, or intersection point and records the
      applicable case, the four segment-wise bounds $u_i$, the junction correction
    $\delta$, and the resulting total bound. Each set in this table excludes its boundary: the Region rows give the cells $\Omega_j$ without their boundaries, the Boundary rows give the edges $\Omega_i\cap\Omega_j$ without their endpoints, and the endpoints appear only in the Point rows.}
    \label{tab:tangents_all_case}\\
    \hline
    \textbf{Case} & \textbf{Type} & \textbf{Set containing $\boldsymbol{q}$}
    & $\boldsymbol{u_1}$ & $\boldsymbol{u_2}$ & $\boldsymbol{u_3}$
    & $\boldsymbol{u_4}$ & $\boldsymbol{\delta}$ & $\boldsymbol{\widehat N}$ \\ \hline
    \endfirsthead
    \multicolumn{9}{c}{\tablename~\thetable\ (continued)}\\
    \hline
    \textbf{Case} & \textbf{Type} & \textbf{Set containing $\boldsymbol{q}$}
    & $\boldsymbol{u_1}$ & $\boldsymbol{u_2}$ & $\boldsymbol{u_3}$
    & $\boldsymbol{u_4}$ & $\boldsymbol{\delta}$ & $\boldsymbol{\widehat N}$ \\ \hline
    \endhead
    \hline
    \multicolumn{9}{r}{Continued on the next page}\\
    \endfoot
    \hline
    \endlastfoot
    A--C & Region & $\Omega_{1}$ & 0 & 0 & 0 & 0 & 0 & 0 \\
    A--C & Region & $\Omega_{2}$ & 0 & 0 & 0 & 1 & 0 & 1 \\
    A--C & Region & $\Omega_{3}$ & 0 & 0 & 1 & 0 & 0 & 1 \\
    A--C & Region & $\Omega_{4}$ & 0 & 0 & 0 & 2 & 0 & 2 \\
    A--C & Region & $\Omega_{5}$ & 0 & 0 & 2 & 0 & 0 & 2 \\
    A--C & Region & $\Omega_{6}$ & 0 & 2 & 0 & 0 & 0 & 2 \\
    A--C & Region & $\Omega_{7}$ & 0 & 0 & 1 & 1 & 0 & 2 \\
    A--C & Region & $\Omega_{8}$ & 0 & 1 & 1 & 0 & 0 & 2 \\
    A--C & Region & $\Omega_{9}$ & 0 & 1 & 0 & 1 & 0 & 2 \\
    A--C & Region & $\Omega_{10}$ & 1 & 1 & 0 & 0 & 0 & 2 \\
    A--C & Region & $\Omega_{11}$ & 1 & 0 & 1 & 0 & 0 & 2 \\
    A--C & Region & $\Omega_{12}$ & 1 & 0 & 0 & 1 & 0 & 2 \\
    A--C & Region & $\Omega_{13}$ & 1 & 0 & 1 & 0 & 0 & 2 \\
    A--C & Region & $\Omega_{14}$ & 0 & 1 & 0 & 1 & 0 & 2 \\
    A--C & Region & $\Omega_{15}$ & 2 & 0 & 0 & 1 & 0 & 3 \\
    A--C & Region & $\Omega_{16}$ & 1 & 1 & 0 & 1 & 0 & 3 \\
    A,C & Region & $\Omega_{15}'$ & 2 & 0 & 1 & 0 & 0 & 3 \\
    A & Region & $\Omega_{6}'$ & 0 & 2 & 0 & 1 & 0 & 3 \\
    B,C & Region & $\Omega_{15}''$ & 2 & 0 & 0 & 0 & 0 & 2 \\ \hline
    A--C & Boundary & $\Omega_{1}\cap\Omega_{4}$ & 0 & 0 & 0 & 1 & 0 & 1 \\
    A--C & Boundary & $\Omega_{1}\cap\Omega_{5}$ & 0 & 0 & 1 & 0 & 0 & 1 \\
    A--C & Boundary & $\Omega_{2}\cap\Omega_{1}$ & 0 & 0 & 0 & 1 & 0 & 1 \\
    A--C & Boundary & $\Omega_{2}\cap\Omega_{3}$ & 0 & 0 & 1 & 1 & 1 & 1 \\
    A--C & Boundary & $\Omega_{2}\cap\Omega_{4}$ & 0 & 0 & 0 & 2 & 0 & 2 \\
    A--C & Boundary & $\Omega_{3}\cap\Omega_{7}$ & 0 & 0 & 1 & 1 & 0 & 2 \\
    A--C & Boundary & $\Omega_{4}\cap\Omega_{7}$ & 0 & 0 & 1 & 2 & 1 & 2 \\
    A--C & Boundary & $\Omega_{5}\cap\Omega_{7}$ & 0 & 0 & 2 & 1 & 1 & 2 \\
    A--C & Boundary & $\Omega_{6}\cap\Omega_{1}$ & 0 & 1 & 0 & 0 & 0 & 1 \\
    A--C & Boundary & $\Omega_{6}\cap\Omega_{8}$ & 0 & 2 & 1 & 0 & 1 & 2 \\
    A--C & Boundary & $\Omega_{7}\cap\Omega_{9}$ & 0 & 1 & 1 & 1 & 1 & 2 \\
    A--C & Boundary & $\Omega_{8}\cap\Omega_{5}$ & 0 & 1 & 2 & 0 & 1 & 2 \\
    A--C & Boundary & $\Omega_{8}\cap\Omega_{9}$ & 0 & 1 & 1 & 1 & 1 & 2 \\
    A--C & Boundary & $\Omega_{10}\cap\Omega_{6}$ & 1 & 2 & 0 & 0 & 1 & 2 \\
    A--C & Boundary & $\Omega_{10}\cap\Omega_{11}$ & 1 & 1 & 1 & 0 & 1 & 2 \\
    A--C & Boundary & $\Omega_{11}\cap\Omega_{8}$ & 1 & 1 & 1 & 0 & 1 & 2 \\
    A--C & Boundary & $\Omega_{12}\cap\Omega_{13}$ & 1 & 0 & 1 & 1 & 1 & 2 \\
    A--C & Boundary & $\Omega_{12}\cap\Omega_{15}$ & 2 & 0 & 0 & 1 & 0 & 3 \\
    A--C & Boundary & $\Omega_{13}\cap\Omega_{3}$ & 1 & 0 & 1 & 0 & 0 & 2 \\
    A--C & Boundary & $\Omega_{14}\cap\Omega_{12}$ & 1 & 1 & 0 & 1 & 1 & 2 \\
    A--C & Boundary & $\Omega_{14}\cap\Omega_{16}$ & 1 & 1 & 0 & 1 & 0 & 3 \\
    A--C & Boundary & $\Omega_{15}\cap\Omega_{2}$ & 1 & 0 & 0 & 1 & 0 & 2 \\
    A--C & Boundary & $\Omega_{16}\cap\Omega_{10}$ & 1 & 1 & 0 & 1 & 0 & 3 \\
    A--C & Boundary & $\Omega_{16}\cap\Omega_{15}$ & 2 & 1 & 0 & 1 & 1 & 3 \\
    A,C & Boundary & $\Omega_{15}'\cap\Omega_{13}$ & 2 & 0 & 1 & 0 & 0 & 3 \\
    A,C & Boundary & $\Omega_{15}'\cap\Omega_{3}$ & 1 & 0 & 1 & 0 & 0 & 2 \\
    A,C & Boundary & $\Omega_{15}'\cap\Omega_{15}$ & 2 & 0 & 1 & 1 & 1 & 3 \\
    A & Boundary & $\Omega_{6}'\cap\Omega_{2}$ & 0 & 1 & 0 & 1 & 0 & 2 \\
    A & Boundary & $\Omega_{6}'\cap\Omega_{6}$ & 0 & 2 & 0 & 1 & 0 & 3 \\
    A & Boundary & $\Omega_{6}'\cap\Omega_{16}$ & 1 & 2 & 0 & 1 & 1 & 3 \\
    B,C & Boundary & $\Omega_{15}''\cap\Omega_{10}$ & 2 & 1 & 0 & 0 & 1 & 2 \\
    B,C & Boundary & $\Omega_{15}''\cap\Omega_{15}$ & 2 & 0 & 0 & 1 & 0 & 3 \\
    B,C & Boundary & $\Omega_{15}''\cap\Omega_{1}$ & 1 & 0 & 0 & 0 & 0 & 1 \\ \hline
    A,C & Point & $\boldsymbol{e}(s_{0})$ & 1 & 0 & 1 & 0 & 0 & 2 \\
    B & Point & $\boldsymbol{e}(s_{0})$ & 1 & 0 & 0 & 1 & 0 & 2 \\
    A--C & Point & $\boldsymbol{e}(s_{\pi})$ & 0 & 1 & 1 & 0 & 1 & 1 \\
    A--C & Point & $\boldsymbol{e}(s_{3\pi/2})$ & 0 & 0 & 1 & 1 & 1 & 1 \\
    A--C & Point & $\boldsymbol{e}(s_{2\pi})$ & 0 & 0 & 0 & 1 & 0 & 1 \\
    A--C & Point & $\Omega_{12}\cap\Omega_{14}\cap\Omega_{15}\cap\Omega_{16}$ & 2 & 1 & 0 & 1 & 1 & 3 \\
    B & Point & $\Omega_{3}\cap\Omega_{12}\cap\Omega_{13}\cap\Omega_{2}$ & 1 & 0 & 1 & 1 & 1 & 2 \\
    A--C & Point & $\Omega_{6}\cap\Omega_{8}\cap\Omega_{10}\cap\Omega_{11}$ & 1 & 2 & 1 & 0 & 2 & 2 \\
    A--C & Point & $\Omega_{5}\cap\Omega_{7}\cap\Omega_{8}\cap\Omega_{9}$ & 0 & 1 & 2 & 1 & 2 & 2 \\
    A--C & Point & $\Omega_{2}\cap\Omega_{3}\cap\Omega_{4}\cap\Omega_{7}$ & 0 & 0 & 1 & 2 & 1 & 2 \\
    B,C & Point & $\boldsymbol{e}(s_{\pi/2})$ & 1 & 1 & 0 & 0 & 1 & 1 \\
    A & Point & $\boldsymbol{e}(s_{\pi/2})$ & 1 & 1 & 0 & 1 & 1 & 2 \\
    A,C & Point & $\Omega_{15}'\cap\Omega_{15}\cap\Omega_{12}\cap\Omega_{13}$ & 2 & 0 & 1 & 1 & 1 & 3 \\
    A,C & Point & $\Omega_{15}'\cap\Omega_{15}\cap\Omega_{2}\cap\Omega_{3}$ & 1 & 0 & 1 & 1 & 1 & 2 \\
    A & Point & $\Omega_{6}'\cap\Omega_{6}\cap\Omega_{1}\cap\Omega_{2}$ & 0 & 1 & 0 & 1 & 0 & 2 \\
    A & Point & $\Omega_{6}'\cap\Omega_{6}\cap\Omega_{16}\cap\Omega_{10}$ & 1 & 2 & 0 & 1 & 1 & 3 \\
    B,C & Point & $\Omega_{15}''\cap\Omega_{15}\cap\Omega_{1}\cap\Omega_{2}$ & 1 & 0 & 0 & 1 & 0 & 2 \\
    B,C & Point & $\Omega_{15}''\cap\Omega_{15}\cap\Omega_{16}\cap\Omega_{10}$ & 2 & 1 & 0 & 1 & 1 & 3 \\
  \end{longtable}

  \normalsize
  Every value of $\widehat N$ in \Cref{tab:tangents_all_case} is at
  most three. Furthermore, the values equal to three lie in the strip
  $x(0)\leq q_x\leq x(s_{2\pi})$, as asserted in
  \Cref{lem:global-tangency-count}.

  \section{The Optimized Algorithm}
  \Cref{alg:candidate-selection} states the resulting reduced
  candidate-selection procedure. In contrast to the original distance
  algorithm, the midpoint initialization is invoked only as a numerical
  fallback; when neither endpoint derivative test activates, \cref{min-max-min} guarantees that the global minimum is attained at an endpoint, and no interior search is needed.
  \par\medskip
  \begingroup
  \refstepcounter{algorithm}
  \label{alg:candidate-selection}
  \noindent\hrule height .8pt depth 0pt \kern 2pt
  \noindent\textbf{Algorithm~\thealgorithm} Candidate selection with numerical fallback
  \par\kern 2pt \hrule \kern 2pt \nopagebreak
  \begin{algorithmic}[1]
  \REQUIRE a proper no-inflection clothoid segment $p:[0,L]\to\mathbb{R}^2$, a query point $q$, an endpoint iteration $\mathcal{S}$
  \ENSURE a candidate $s_\star$ for the global minimum of $D(s)=\tfrac12\|p(s)-q\|^2$ on $[0,L]$, or \textsc{NumericalFailure}
  \IF{$D(0)=0$} \RETURN $0$ \ENDIF
  \IF{$D(L)=0$} \RETURN $L$ \ENDIF
  \STATE $s_0\gets0$;\quad $s_1\gets L$;\quad $\mathrm{fb}\gets\mathrm{false}$
  \STATE $f_0\gets\bigl(D'(0)<0\bigr)$;\quad $f_1\gets\bigl(D'(L)>0\bigr)$
    \COMMENT{endpoint activity tests}
  \IF{$f_0$}
    \STATE $[s_0,f_0]\gets\mathcal{S}(0)$;\quad $\mathrm{fb}\gets\mathrm{fb}\lor\neg f_0$
  \ENDIF
  \IF{$f_1$}
    \STATE $[s_1,f_1]\gets\mathcal{S}(L)$;\quad $\mathrm{fb}\gets\mathrm{fb}\lor\neg f_1$
  \ENDIF
  \IF{$\mathrm{fb}$}
    \STATE $[s_m,f_m]\gets\mathcal{S}(L/2)$ \COMMENT{A numerical fallback as the iterations fail}
    \IF{$f_m$}
      \RETURN $\arg\min_{s\in\{0,L,s_0,s_1,s_m\}}D(s)$
    \ELSE
      \RETURN \textsc{NumericalFailure}
    \ENDIF
  \ENDIF
  \RETURN $\arg\min_{s\in\{0,L,s_0,s_1\}}D(s)$
    \COMMENT{reduces to $\{0,L\}$ when $f_0=f_1=\mathrm{false}$ by \cref{min-max-min}}
  \end{algorithmic}
  \kern 2pt \hrule
  \endgroup
  \par\medskip
  \normalsize

  \bibliographystyle{siamplain}
  \bibliography{references}

  \end{document}